\documentclass[11pt,a4paper]{amsart}

\usepackage{amssymb,amscd,amsmath,young, stmaryrd,tikz-cd}
\tikzcdset{every label/.append style = {font = \normalsize}}
\usepackage{mathrsfs, mathdots}
\usepackage[mathcal]{eucal}
\usepackage{float,mathabx}
\usepackage{soul,caption}
\usepackage{cancel}
\usepackage[normalem]{ulem}
\usepackage{hyperref, wasysym}
\usetikzlibrary{arrows}
\usepackage{longtable,array,multirow,makecell,graphicx}

\theoremstyle{plain}
\newtheorem{thm}{Theorem}[section]
\newtheorem{lem}[thm]{Lemma}
\newtheorem{pro}[thm]{Proposition}
\newtheorem{cor}[thm]{Corollary}
\newtheorem*{claim*}{Claim}

\newtheorem{cond}[thm]{Condition}

\theoremstyle{remark}
\newtheorem{ass}[thm]{Assumption}
\newtheorem{rem}[thm]{Remark}

\newtheorem{exm}[thm]{Example}

\newtheorem{dfn}[thm]{Definition}
\newtheorem*{acknowledgements}{Acknowledgements}

\numberwithin{equation}{section}

\numberwithin{table}{section}

\newcommand{\lcard}{l}

\newcommand{\N}{\mathbb{N}}
\newcommand{\Z}{\mathbb{Z}}
\newcommand{\Q}{\mathbb{Q}}

\newcommand{\C}{\mathbb{C}}

\newcommand{\mff}{\mathfrak{f}}

\newcommand{\mfh}{\mathfrak{h}}
\newcommand{\mfp}{\mathfrak{p}}
\newcommand{\mfP}{\mathfrak{P}}

\newcommand{\lri}{\mathfrak{o}}
\newcommand{\lrispec}{\mathfrak{o}}
\newcommand{\Lri}{\mathfrak{O}}

\newcommand{\Gri}{\ensuremath{\mathcal{O}}}

\renewcommand{\epsilon}{\varepsilon}

\renewcommand{\phi}{\varphi}
\renewcommand{\theta}{\vartheta}

\newcommand{\mcO}{\mathcal{O}}

\newcommand{\mcE}{\mathcal{E}}

\newcommand{\msfF}{\mathsf{F}}
\newcommand{\msfK}{\mathsf{K}}
\newcommand{\msfQ}{\mathsf{Q}}
\newcommand{\msfL}{\mathsf{L}}
\newcommand{\msfS}{\mathsf{S}}

\newcommand{\rarr}{\rightarrow}

\newcommand{\bmcL}{\boldsymbol{\mathcal{L}}}
\newcommand{\bmcV}{\boldsymbol{\mathcal{V}}}

\newcommand{\bse}{\boldsymbol{e}}
\newcommand{\bsLambda}{\boldsymbol{\Lambda}}
\newcommand{\bsGamma}{\boldsymbol{\Gamma}}

\newcommand{\Zp}{\mathbb{Z}_{p}}

\newcommand{\SubRep}{\textup{SubRep}}
\newcommand{\SubMod}{\textrm{SubMod}}
\newcommand{\bsZ}{\boldsymbol{Z}}
\newcommand{\bsM}{\boldsymbol{M}}
\newcommand{\bsdelta}{\boldsymbol{\delta}}
\newcommand{\mcM}{\mathcal{M}}
\newcommand{\msfh}{\mathsf{h}}
\newcommand{\msft}{\mathsf{t}}
\newcommand{\ps}{{P}}
\newcommand{\pa}{\mathbf{P}}

\DeclareMathOperator{\ima}{im}

\DeclareMathOperator{\Fil}{Fil}
\DeclareMathOperator{\maj}{maj}
\DeclareMathOperator{\des}{des}
\DeclareMathOperator{\diag}{diag}
\DeclareMathOperator{\End}{End}
\DeclareMathOperator{\id}{id}
\DeclareMathOperator{\Spec}{Spec}
\DeclareMathOperator{\gr}{gr}

\DeclareMathOperator{\ad}{ad}

\DeclareMathOperator{\rk}{rk}

\DeclareMathOperator{\GL}{GL}

\DeclareMathOperator{\Hom}{Hom}
\DeclareMathOperator{\Mat}{Mat}

\DeclareMathOperator{\Id}{Id}

   \def \la {\langle} \def \ra {\rangle} \def \grL {\textup{gr}L} \def
\idealgr {\triangleleft_\textup{gr}}

\def \bfalpha {\boldsymbol{\alpha}}

\def \bfB {{\bf B}}
\def \bfG {{\bf G}}

\def \bfc {{\bf c}}

\def \bff {{\bf f}}

\def \bfg {{\bf g}}
\def \bfm {{\bf m}}

\def \bfr {{\bf r}}
\def \bfR {{\bf R}}
\def \bfs {{\bf s}}
\def \bfS {{\bf S}}
\def \bfx {{\bf x}}
\def \bfI {{\boldsymbol{I}}}
\def \bfJ {{\boldsymbol{J}}}
\def \bfR {{\boldsymbol{R}}}
\def \bfLambda {\boldsymbol{\Lambda}}
\def \bfX {{\bf X}}
\def \bfy {{\bf y}}
\def \bfY {{\bf Y}}
\def \wt {\widetilde}

\def \mcC {\ensuremath{\mathcal{C}}}

\def \mcL {\ensuremath{\mathcal{L}}}

\def \mcV {\ensuremath{\mathcal{V}}}

\def \Fp {\ensuremath{\mathbb{F}_p}}

\def \mcP {\ensuremath{\mathcal{P}}}
\def \mcR {\ensuremath{\mathcal{R}}}

\def \mfo {\ensuremath{\mathfrak{o}}}

\def \Fq {\ensuremath{\mathbb{F}_q}}
\def \Zp  {\mathbb{Z}_p}

\def \Mat {\mathrm{Mat}}
\def \bsy {{\bf y}}
\def \bI {\boldsymbol{I}}
\def \bJ {\boldsymbol{J}}
\def \br {\boldsymbol{r}}
\def \bR {\boldsymbol{R}}

\author{Seungjai Lee} \address{The Research Institute of Basic Sciences, Seoul National University, Seoul 08826, South Korea}
\email{seungjai.lee@snu.ac.kr}

\author{Christopher Voll} \address{Fakult\"at f\"ur Mathematik,
  Universit\"at Bielefeld, D-33501 Bielefeld, Germany}
\email{C.Voll.98@cantab.net}


\begin{document}

\title[Zeta functions of integral nilpotent quiver representations]{Zeta
  functions of integral nilpotent quiver representations}

\date{\today}
\begin{abstract}
  We introduce and study multivariate zeta functions enumerating
  subrepresentations of integral quiver representations. For nilpotent
  such representations defined over number fields, we exhibit a
  homogeneity condition that we prove to be sufficient for local
  functional equations of the generic Euler factors of these zeta
  functions. This generalizes and unifies previous work on submodule
  zeta functions including, specifically, ideal zeta functions of
  nilpotent (Lie) rings and their graded analogues.
\end{abstract}

\maketitle

\thispagestyle{empty}

\setcounter{tocdepth}{1} \tableofcontents{}

\thispagestyle{empty}

\section{Introduction}\label{sec:intro}

\subsection{Quivers, (sub-)representations, and zeta functions}
Representations of quivers over fields have been extensively studied;
see, e.g., \cite{WCB/06,DerksenWeyman/17}. Of particular importance in
this extensive field is the geometry of \emph{quiver Grassmannians},
viz.\ algebraic varieties parameterizing the subrepresentations of a
given quiver representation with a fixed dimension vector
(e.g.\ \cite{IFR/13}). Despite their arithmetic significance
(e.g.\ \cite{WCB/96}), representations of quivers over rings seem to
have received much less attention.  In this paper we study
multivariate zeta functions enumerating the subrepresentations of
integral representations of finite quivers. In the current section we
define these terms.

A \textit{quiver} is a directed graph, formally a quadruple
$\msfQ=(Q_{0},Q_{1},\msfh,\msft)$, where $Q_{0}$ and $Q_{1}$ are
finite sets, called the set of \textit{vertices} and \textit{arrows}
of $\msfQ$, respectively, and $\msfh,\msft:Q_{1}\rightarrow Q_{0}$ are
maps assigning to each arrow its \textit{head} and \textit{tail},
respectively. Arrows $\phi \in Q_1$ with $\msft(\phi)=\msfh(\phi)$ are
called \textit{loops}.

For a ring $R$, an $R$-\emph{representation}
\[ V = V_{\msfQ}=\left(\mcL_{\iota},f_{\phi}\right)_{\iota\in Q_{0},\phi\in
    Q_{1}} \] of $\msfQ$ consists of a family of $R$-modules $\mcL_{\iota}$
indexed by the vertices $\iota\in Q_{0}$ and a family of linear maps
$f_{\phi}:\mcL_{\msft(\phi)}\rightarrow\mcL_{\msfh(\phi)}$ indexed by the
arrows $\phi\in Q_{1}$. We call $V$ \emph{finitely generated} resp.\
\emph{free of finite rank} if all the $R$-modules $\mcL_\iota$ are finitely
generated resp.\ free of finite rank (which may vary with $\iota\in Q_0$). A
representation over a ring of integers of a global or local field is called
\emph{integral}. Given an $R$-algebra~$S$, we write
\[V(S)=V\otimes_RS = \left(\mcL_{\iota}\otimes_R S, f_\phi\otimes_R
    \id_S\right)\] for the resulting $S$-representation of~$\msfQ$. If $S$ is
also an $R'$-algebra, we write $V(S)_{R'}$ for the corresponding restriction
of scalars. If $R=\Gri$, the ring of integers of a global number field, and
$\mfp\in\Spec(\Gri)\setminus\{0\}$ is a non-zero prime ideal of $\Gri$, then
we write $\Gri_{\mfp}$ for the completion of $\Gri$ at~$\mfp$.

      If $V$ is an $R$-representation of $\msfQ$ which is free of finite rank,
      say $n_{\iota} = \rk_R\mcL_{\iota}$ for $\iota\in Q_0$, we call
      $(n_{\iota})_{\iota\in Q_{0}}$ the \textit{rank vector} and
      $\rk V = \sum_{\iota\in Q_{0}}n_{\iota}$ the (\textit{total})
      \textit{rank} of $V$.  If $n_{\iota}=1$ for all $\iota\in Q_0$, then $V$
      is said to be \textit{thin}.

      Let $R$ be a ring and $V=\left(\mcL_{\iota},f_{\phi}\right) $ be an
      $R$-re\-pre\-sen\-ta\-tion of~$\msfQ$.  Given submodules
      $\Lambda_{\iota} \leq \mcL_{\iota}$ for all $\iota\in Q_0$ such that
      $f'_{\phi}(\Lambda_{\msft(\phi)})\leq \Lambda_{\msfh(\phi)}$ for all
      $\phi\in Q_1$, where $f'_\phi := f_\phi|_{\Lambda_{t(\phi)}}$, we call
  \[V'=(\Lambda_{\iota},f'_{\phi})_{\iota\in Q_{0},\phi\in Q_{1}}\] a
  \textit{subrepresentation} of~$V$, written~$V' \leq V$.  We call
  $V'$ of \emph{finite index} in $V$ if
  $\left|\mcL_{\iota}:\Lambda_{\iota}\right|<\infty$ for all $\iota\in
  Q_{0}$. 
  

  For $\bfm=(m_\iota)_{\iota\in Q_0}\in\N^{Q_0}$, write
  \[a_{\bfm}(V)=\#\left\{V'\leq_{f}V\mid\forall \iota\in Q_{0}, |\mcL_{\iota}:\Lambda_{\iota}|=m_{\iota}\right\}\]
  for the number of subrepresentations of $V$ of \emph{index}~$\bfm$.
  
  Assume that $a_{\bfm}(V)<\infty$ for all $\bfm$. The
  (\emph{multivariate representation}) \emph{zeta function} associated
  with $V$ is the Dirichlet generating series
  \begin{align}
  	\zeta_{V}(\bfs)&=\sum_{V'\leq_{f}V}\prod_{\iota\in\Q_{0}}|\mcL_{\iota}:\Lambda_{\iota}|^{-s_{\iota}}=\sum_{\bfm\in\N^{Q_0}}a_{\bfm}(V)\prod_{\iota\in Q_0}m_{\iota}^{-s_{\iota}},\label{def:zeta}
  \end{align}
  enumerating the finite-index subrepresentations of~$V$. Here each
  $s_{\iota}$ is a complex variable and~$\bfs=(s_{\iota})_{\iota\in
    Q_0}$.
  
  We will sometimes turn to a specific univariate specialization of
  $\zeta_V(\bfs)$. We set
  $$|V:V'|=\prod_{\iota\in Q_{0}}|\mcL_{\iota}:\Lambda_{\iota}| < \infty.$$
  For $m\in\N$, write
  \[a_{m}(V)=\#\left\{V'\leq_{f}V\mid|V:V'|=m\right\} = \sum_{\left\{\bfm\mid \sum_\iota m_\iota = m\right\}} a_{\bfm}(V)\]
  for the number of subrepresentations of $V$ of index~$m$. The
  (\emph{univariate representation}) \emph{zeta function} associated
  with $V$ is the Dirichlet generating series
  \begin{equation*}
    \zeta_{V}(s)=\sum_{m\geq1}a_m(V)m^{-s} = \zeta_{V}(s,\dots,s)
  \end{equation*}
 obtained by substitution all variables $s_\iota$ with a single
 complex variable~$s$. We trust that the slight abuse of notation
 caused by denoting both the uni- and multivariate functions by
 $\zeta_V$ will not cause confusion. Note that, trivially, multi- and
 univariate representation zeta functions coincide in the case of loop
 quivers, i.e.\ when~$|Q_0|=1$; cf.\ Section~\ref{subsubsec:submod}.

As we shall explain in Section~\ref{sec:apps}, the class of univariate
representation zeta functions associated with quiver representations
coincides with the class of \emph{submodule zeta functions},
enumerating sublattices invariant under a set of linear
operators. This class of zeta function was pioneered by Solomon
(\cite{Solomon/77}) and, much more recently, further developed by
Rossmann (\cite{RossmannTAMS/18}; see
Section~\ref{subsubsec:submod}). We argue, however, that the
interpretation in terms of quiver representations affords new
perspectives on (graded) submodule zeta functions, even in the
univariate case.

We will soon focus on representations $V$ where the numbers $a_{m}(V)$
are finite for all~$m$ and, moreover, grow at most polynomially. This
holds in particular if $R$ is the ring of integers of a global or
local field and $V$ is free of finite rank. Under these assumptions
the formal Dirichlet generating series $\zeta_V(s)$ converges
absolutely for all $s\in\C$ with $\Re(s)> \rk V$. Indeed, it is clear
that then $a_m(V)$ is bounded above by the number of sublattices of
$R^{\rk V}$ of index~$m$; the claim follows from the well-known facts
recalled in Example~\ref{exa:ab}. For further properties of univariate
representation zeta functions of quiver representations, see
Section~\ref{subsubsec:ana}.

If $R=\Gri$ is a ring of integers of a global field and $V$ is
finitely generated, then its zeta function satisfies the formal Euler
product
\begin{equation}\label{equ:euler.prod}
  \zeta_{V}(\bfs)=\prod_{\mfp\in\Spec(\Gri)\setminus\{(0)\}}\zeta_{V(\Gri_{\mfp})}(\bfs).
\end{equation}
It follows from deep general results that each of the local zeta
functions $\zeta_{V(\Gri_{\mfp})}(\bfs)$ is a rational function in
$q^{-s_\iota}$, $\iota\in Q_0$, where $q$ is the residue field
cardinality of the compact discrete valuation ring~$\Gri_{\mfp}$. In
the univariate case this follows essentially from~\cite{GSS/88}; the
arguments extend to the multivariate situation and the case of general
number fields.

\begin{exm}\label{exa:ab}
  Assume that $Q_0$ is a singleton, represented by an $R$-module
  $\mcL$, and $f_\phi=0$ for all $\phi\in Q_1$. The representation
  zeta function $\zeta_V(s)$ simply enumerates the $R$-submodules of
  finite index in $\mcL$. If further $R=\Gri$, the ring of integers of
  a number field~$K$, and $\mcL\cong\Gri^n$, it is well-known (see,
  e.g., \cite[Proposition~1.1]{GSS/88}) that
     $$\zeta_V(s) = \prod_{i=0}^{n-1}\zeta_K(s-i),$$ where
 $\zeta_K(s)$ is the Dedekind zeta function of $K$. Note that $\zeta_V(s)$ converges
 \emph{precisely} if $\Re(s)>n=\rk V$, as the abscissa of convergence
 of $\zeta_K(s)$ is equal to~$1$.

Locally, i.e.\ when $R=\lri$ is a compact discrete valuation ring (cDVR) of residue field cardinality
$q$, we obtain (with $t=q^{-s}$)
\begin{equation}\label{equ:zeta.null}
  \zeta_{V(\lri)}(s) = \zeta_{\lri^n}(s) =
  \prod_{i=0}^{n-1}\frac{1}{1-q^{i}t}.
\end{equation}
\end{exm}

\begin{exm}
  Let $L$ be a Lie ring (i.e.\ $\Z$-Lie algebra), with (Lie)
  generators $x_1,\dots,x_d$. We may interpret $L$ as a representation
  of the loop quiver $\msfQ = \msfL_d$ with one vertex and $d$
  loops. Indeed, if $Q_0$ is a singleton and $|Q_1|=d$, then $V = (L,
  \ad x_1,\dots,\ad x_d)$ is a $\Z$-representation of~$\msfL_d$. The
  zeta function of $V$ coincides with the \emph{ideal zeta function}
  of $L$ introduced in~\cite{GSS/88};
  cf.\ Section~\ref{subsubsec:submod}. Example~\ref{exa:ab} is the
  special case of the abelian Lie ring.
\end{exm}

\subsection{Nilpotent quiver representations and the homogeneity condition}\label{subsec:nil.qui}
Our main result concerns integral nilpotent quiver representations
satisfying a certain homogeneity condition, which we now
explain. Recall that $\msfQ=(Q_0,Q_1, \msfh,\msft)$ is a quiver and
$V=\left(\mcL_{\iota},f_{\phi}\right)$ is an $R$-representation
of~$\msfQ$.

A \textit{path} $w=\phi_{1}\cdots\phi_{n}$ in $\msfQ$ of length $n\geq1$ is a
sequence of arrows $\phi_{1},\ldots,\phi_{n}$ such that
$\msft(\phi_{i})=\msfh(\phi_{i+1})$ for $1\leq i \leq n-1$. One calls
$\msfh(w)=\msfh(\phi_{1})$ the \textit{head} of $w$ and
$\msft(w)=\msft(\phi_{n})$ the \textit{tail} of $w$. We also say that $w$ is a
path from $\msft(w)$ to $\msfh(w)$.  Any vertex $x\in Q_0$ of $\msfQ$ is
considered as a path of length 0 with head $x$ and tail~$x$, and denoted
by~$w_x$.

For a path $w=\phi_1\cdots\phi_n$, write $f_{w}=f_{\phi_1}\cdots f_{\phi_n}$
and set
$$w(V) = f_{w}(V) = f_w(\mcL_{t(w)}) \leq \mcL_{h(w)}.$$ If $w_{x}$ is
a path of length $0$, then $w_x(V)=\mcL_{x}$.

We say that $V$ is \textit{nilpotent} if there exists $t\in\N$ such that
$w(V)=0$ for any path $w$ of length $t$. If $V$ is nilpotent, then the unique
$c\in\N$ such that $w(V)=0$ for any path $w$ of length $c$ but not for all
paths of length $c-1$ is called the \textit{nilpotency class} of~$V$.  Note
that our definition of nilpotency ensures that all representations of
\emph{acyclic} quivers, viz.\ quivers without oriented cycles, are nilpotent.
  
Set
\begin{equation}\bmcL:=\bigoplus_{\iota\in
  Q_{0}}\mcL_{\iota}.\label{equ:bmcL}
\end{equation}
Let $\delta_{i,j}$ denote the usual Kronecker delta. For any $\phi\in Q_1$ we
extend $f_\phi\in\Hom(\mcL_{\msft(\phi)},\mcL_{\msfh(\phi)})$ to an
endomorphism $e_{\phi}\in\End(\bmcL)$ by setting, for $\iota\in Q_0$,
$$e_\phi|_{\mcL_{\iota}} = \delta_{\iota,\msft(\phi)} f_\phi.$$
Informally speaking, we obtain the endomorphism $e_\phi$ by trivial extension
of the homomorphism~$f_\phi$. Set
\begin{equation*}\label{def:end.alg}
  \mcE = \mcE(V) := \langle e_{\phi} \mid \phi\in Q_{1}\rangle
  \leq \End(\bmcL).\end{equation*} We recursively define the
\emph{upper centralizer series of $V$}, viz.\ the flag
$$\left(Z_{i}\right)_{i\in\N_0}=\left(\left(Z_{\iota,i},f_{\phi}|_{Z_{\iota,i}}\right)_{\iota\in
  Q_{0},\phi\in Q_{1}}\right)_{i\in\N_0}$$ of subrepresentations of
$V$, by setting $Z_{0}={0}$ and, for $i\in\N_0$,
\[Z_{i+1}/Z_{i}=\textrm{Cent}_{\mcE}(V/Z_{i}):=\left\{x+Z_{i}\in V/Z_{i}\mid x \mcE\subseteq Z_{i}\right\}\]
and $Z_{\iota,i}:= Z_i \cap \mcL_\iota$.  One can check that $V$ is
nilpotent if $Z_{k}=V$ for some $k\in\N_0$. In this case, the number
$c=c(V)=\min\{k\in\mathbb{N}_{0}\mid Z_{k}=V\}$ is the nilpotency
class of $V$ (or, equivalently, of $\mcE$).

In this article we consider finitely generated integral nilpotent
quiver representations $V$ over a global ring of integers $R=\Gri$,
say, satisfying the following assumption.
\begin{ass}\label{ass:quiver.cocentral}
  For each $\iota\in Q_{0}$, there exist free $\Gri$-submodules
  $\mcL_{\iota,1},\ldots,\mcL_{\iota,c} \leq \mcL_{\iota}$ such that
	\[\mcL_{\iota}=\bigoplus_{j=1}^c\mcL_{\iota,j} \quad \textup{ and }\quad Z_{\iota,i}=\bigoplus_{j>c-i}\mcL_{\iota,j}.\]
\end{ass}

\begin{rem}
  Assumption~\ref{ass:quiver.cocentral} is closely analogous to
  \cite[Assumption~1.1]{VollIMRN/19}. As the latter, it is satisfied
  automatically if $\Gri$ is a principal ideal domain. It is only made
  for notational convenience; see \cite[Remark~1.1]{VollIMRN/19}.
\end{rem}

For $i\in[c]_0=\{0,1,2,\dots,c\}$ we set
\begin{alignat*}{2}
\bmcL_{i}&:=\bigoplus_{\iota\in
  Q_{0}}\mcL_{\iota,i}\nonumber&\quad\textup{ and }\quad
 \bsZ_{i}&:=\bigoplus_{\iota\in
          Q_{0}}Z_{\iota,i}.\nonumber
\end{alignat*}

We further set $\mcL_{\iota,0}=\mcL_{\iota,c+1}=\{0\}.$

For $\iota\in Q_{0}$ and $i\in[c]_{0}$, let $n=\rk V$, $n_{\iota}=\rk_{\Gri}\mcL_{\iota}$, $n_{\iota,i}=\rk_{\Gri}\mcL_{\iota,i}$,
\[N_{\iota,i} = \rk_{\Gri} \bigoplus_{j\leq c-i}\mcL_{\iota,j}=\sum_{j\leq
    c-i}n_{\iota,j}=\rk_{\Gri}(\mcL_{\iota}/Z_{\iota,i}),\] and
\[N_{i}=\sum_{\iota\in
    Q_{0}}N_{\iota,i}=\rk_{\Gri}(\bmcL/\bsZ_{i}).\] Note that
\[n=N_{0}=\sum_{\iota\in Q_{0}}n_{\iota}=\sum_{\iota,i}n_{\iota,i}.\]

Our main results concern generic Euler factors of zeta functions of quiver
representations which satisfy a certain condition. The following
generalizes~\cite[Condition~1.1]{VollIMRN/19}.

\begin{cond}[homogeneity]\label{cond:quiver.hom}
  The nilpotent associative algebra
  $\mcE=\mcE(V)\subseteq\textnormal{End}_{\Gri}(\bmcL)$ is generated by
  elements $c_{1},\ldots,c_{d}$ such that, for all $k\in[d]=\{1,\dots,d\}$ and
  $j\in[c]$,
	\[\bmcL_{j}c_{k}\subseteq\bmcL_{j+1}.\]
\end{cond}

\begin{rem}
 As in \cite{VollIMRN/19}, Condition~\ref{cond:quiver.hom} is
 satisfied if $\mcE$ is cyclic (i.e.\ one may chose $d=1$) or if
 $c\leq 2$. It is also stable under taking direct sums of
 representations. The impact of this fundamental operation on the
 associated zeta function is, in general, poorly understood.
\end{rem}

The following is our main theorem.
\begin{thm}\label{thm:main.refine}
	Assume that $\mcE=\mcE(V) \subseteq \End_{\Gri}(\mcL)$
satisfies Condition~\ref{cond:quiver.hom}.  Then, for almost all prime
ideals~$\mfp$ of $\Gri$ and all finite extensions $\Lri$ of
$\Gri_{\mfp}$, with residue field cardinality~$q^f$, the following
functional equation holds:
	\begin{equation}\label{equ:funeq.refine}
		\left.\zeta_{V(\Lri)}(\bfs)\right|_{q\rightarrow
			q^{-1}}=(-1)^{n}q^{f\left(\left(\sum_{\iota\in
				Q_{0}}\binom{n_{\iota}}{2}\right)-\left(\sum_{\iota\in Q_{0}}s_{\iota}\left(\sum_{i=0}^{c-1}N_{\iota,i}\right)\right)\right)}\zeta_{V(\Lri)}(\bfs).
	\end{equation}
In particular,
\begin{equation}\label{equ:funeq.refine.univariate}
  \left.\zeta_{V(\Lri)}(s)\right|_{q\rightarrow
    q^{-1}}=(-1)^{n}q^{f\left(\left(\sum_{\iota\in
      Q_{0}}\binom{n_{\iota}}{2}\right)-s\left(\sum_{i=0}^{c-1}N_{i}\right)\right)}\zeta_{V(\Lri)}(s).
  \end{equation}

\end{thm}

As in \cite{VollIMRN/19}, a version of the model-theoretic transfer
principle (\cite{CluckersLoeser/10}) implies the following immediate
consequence in positive characteristic.

\begin{cor}\label{cor:pos.char}
 For almost all prime ideals $\mfp$ of $\Gri$ and all finite
 extensions $\Lri$ of~$\Gri_{\mfp}$, with maximal ideal $\mfP$ and
 residue field cardinality $|\Lri/\mfP| = q^f$, say, the following
 functional equation holds:
 \begin{equation}\label{equ:funeq.pos.char}\left.\zeta_{V(\Lri/\mfP\llbracket T
   \rrbracket)}(\bfs)\right|_{q\rightarrow q^{-1}} =
   (-1)^{n}q^{f\left(\left(\sum_{\iota\in
				Q_{0}}\binom{n_{\iota}}{2}\right)-\left(\sum_{\iota\in Q_{0}}s_{\iota}\left(\sum_{i=0}^{c-1}N_{\iota,i}\right)\right)\right)}\zeta_{V(\Lri/\mfP\llbracket
     T \rrbracket)}(\bfs).
 \end{equation}
 In particular,
  \begin{equation*}\label{equ:funeq.pos.char.univariate}\left.\zeta_{V(\Lri/\mfP\llbracket T
   \rrbracket)}(s)\right|_{q\rightarrow q^{-1}} =
    (-1)^{n}q^{f\left(\left(\sum_{\iota\in
        Q_{0}}\binom{n_{\iota}}{2}\right)-s\left(\sum_{i=0}^{c-1}N_{i}\right)\right)}\zeta_{V(\Lri/\mfP\llbracket
      T \rrbracket)}(s).
  \end{equation*}

\end{cor}

\begin{exm}\label{exa:c=1}
  For $c=1$ we have $\mcE=0$, whence the homogeneity
  condition~\ref{cond:quiver.hom} is trivially satisfied. The
  functional equations~\eqref{equ:funeq.refine}
  resp.\ \eqref{equ:funeq.pos.char} follow trivially from inspection
  of the formula
$$\zeta_{V(\lri)}(\bfs) = \prod_{\iota\in Q_0}
  \zeta_{\lri^{n_\iota}}(s_\iota)$$ (cf.\ \eqref{equ:zeta.null}),
  valid for any cDVR $\lri$, regardless of its characteristic.
\end{exm}

\begin{rem}
  The operation $q \rightarrow q^{-1}$ in \eqref{equ:funeq.refine}
  calls for some explanation. If there exists a single rational
  function $W(X,\bfY)\in\Q(X,(Y_\iota)_{\iota\in Q_0})$ such that
  $\zeta_{V(\Lri)}(\bfs) = W(q,q^{-\bfs})$ for cDVRs $\Lri$ whose
  residue characteristics avoid a finite number of primes (depending
  on~$V$), where $q^{-\bfs}=(q^{-s_\iota})_{\iota\in Q_0}$, then the
  functional equation \eqref{equ:funeq.refine} means that
  $$W(X^{-1},\bfY^{-1}) = (-1)^a X^b\bfY^{\bfc} \, W(X,\bfY)$$ for
  suitable $a,b\in\N_0$ and $\bfc\in\N_0^{Q_0}$. It is easy to exhibit
  small examples of quiver representations which violate this
  hypothesis; see, for instance, Proposition~\ref{pro:kron2}. A
  general interpretation of the symmetry expressed
  in~\eqref{equ:funeq.refine} therefore refers to \emph{Denef-type
    formulae} for the zeta functions $\zeta_{V(\Lri)}(\bfs)$,
  viz.\ finite sums involving rational functions
  $W_i(X,\bfY)\in\Q(X,\bfY)$ as above, but also the numbers of
  $\Lri/\mfP$-rational points of (the reductions modulo $\mfp$ of)
  finitely many smooth projective varieties associated with the
  representation~$V$. The precise definition of the operation
  $q\rightarrow q^{-1}$ involves the inversion of the Frobenius
  eigenvalues whose alternating sums yield the relevant numbers of
  rational points, by the Weil conjectures. That this operation is
  well-defined, i.e.\ independent of the choice of Denef-type formula,
  follows from (the straightforward multivariate refinements of the
  arguments given in) \cite[Section~4]{RossmannMPCPS/18}. We refer to
  \cite[Remark~1.7]{VollIMRN/19} for further details.
\end{rem}

\begin{rem}\label{rem:base.ext}
  The finitely many prime ideals $\mfp$ of $\Gri$ we are forced to
  disregard in Theorem~\ref{thm:main.refine} are essentially those for
  which a chosen principalization of ideals of a---in general very
  complicated---algebraic variety has bad reduction modulo~$\mfp$;
  cf.\ Section~\ref{subsec:new.blue}. We know of no bounds on the size
  or shape of this finite set of ``bad'' prime ideals. Examples show,
  however, that it is not just an artefact of our method of proof, but
  non-empty in general.

  Condition~\ref{cond:quiver.hom} is stable under taking direct
  products; see also \cite[Remark~1.8]{VollIMRN/19}. Our methodology
  seems to give us no handle, however, on the sets of prime ideals
  which are bad for a direct product of representations in terms of
  the sets of bad prime ideals of the factors involved. Interesting
  specific questions arise in the context of base extension and
  restriction of scalars of quiver representations over global rings
  of integers. Assume, to be specific, that $\Gri \hookrightarrow
  \Gri'$ is an extension of global rings of integers. Recall that by
  $V(\Gri')_\Gri$ we denote the restriction of scalars to $\Gri$ of
  the $\Gri'$-representation $V(\Gri') = V \otimes_\Gri \Gri'$
  obtained by extension of scalars. For a non-zero prime ideal
  $\mfp\in\Spec(\Gri)$, write $\mfp\Gri'= \prod_{i=1}^g
  \mfP_i^{e_i}$. Then $V(\Gri')\otimes_{\Gri}\Gri_\mfp =
  \bigtimes_{i=1}^g V(\Gri'_{\mfP_i})_{\Gri_\mfp}$, where
  $V(\Gri'_{\mfP_i})_{\Gri_\mfp}$ denotes---in analogy to the
  above---the restriction of scalars to $\Gri_\mfp$ of the
  $\Gri'_{\mfP_i}$-representation $V(\Gri'_{\mfP_i})$.  Thus
  \begin{equation}\label{equ:base.change}
    \zeta_{V(\Gri')_\Gri}(\bfs) = \prod_{\mfp \in
      \Spec(\Gri)\setminus\{(0)\}}\zeta_{V(\Gri')\otimes_\Gri\Gri_{\mfp}}(\bfs) =
    \prod_{\mfp \in
      \Spec(\Gri)\setminus\{(0)\}}\zeta_{\bigtimes_{i=1}^g V(\Gri'_{\mfP_i})_{\Gri_\mfp}}(\bfs).\end{equation}

  The complexity of the Euler factors in~\eqref{equ:base.change} may
  grow dramatically with the degree of the extension $\Gri'/\Gri$,
  even for ``small'' representations~$V$, such as the one in
  Example~\ref{exa:heis}. Large further classes of explicit examples
  are covered in
  \cite[Theorem~1.2]{CSV/19}(=\cite[Theorem~2.2]{CSV_FPSAC2020}),
  albeit only in the univariate case. In general, it seems plausible
  that the bad prime ideals of $\zeta_{V(\Gri')_{\mfp}}(\bfs)$ are
  either ramified in $\Gri'/\Gri$ or lie above bad primes
  of~$\zeta_{V(\Gri)}(\bfs)$.
\end{rem}

\subsection{Submodule zeta functions}\label{sec:apps}
In this section we explain how univariate quiver representation zeta
functions may, equivalently, be seen as submodule zeta functions, and
how Theorem~\ref{thm:main.refine} generalizes previous results about
them. We put particular focus on (\emph{graded}) \emph{ideal zeta
  functions} of ((anti-)commutative) rings. As arguably these classes
of zeta functions arise from the most natural counting problems to
which the framework of quiver representation zeta functions applies,
these provide natural motivation for the vantage point developed in
this paper.

\subsubsection{Submodule zeta functions and loop quiver representations}\label{subsubsec:submod}

Let $\mathcal{V}$ be a finitely generated module over a ring $R$ and
$\Omega\subseteq \End_R(\mathcal{V})$ be a set of $R$-endomorphisms of
$\mathcal{V}$; see \cite[Section~2.2]{RossmannTAMS/18}. The
(\emph{submodule}) \emph{zeta function of $\Omega$ acting on
  $\mathcal{V}$} is the Dirichlet generating series enumerating the
finite-index $\Omega$-invariant submodules $\mathcal{U}$ of
$\mathcal{V}$:
$$\zeta_{\Omega \curvearrowright \mathcal{V}}(s) =
\sum_{\mathcal{U}\leq \mathcal{V}} |\mathcal{V}:\mathcal{U}|^{-s}.$$
If the associative algebra generated by $\Omega$ in
$\End_R(\mathcal{V})$ can be generated by $d$ elements
$c_1,\dots,c_d$, then clearly $\zeta_{\Omega \curvearrowright
  \mathcal{V}}(s) = \zeta_{(\mathcal{V},(c_i)_{i=1}^d)}(s)$. Submodule
zeta functions are therefore representation zeta functions of
\emph{loop quivers}, viz.\ quivers $\msfL_d$ with one vertex and $d$
loops. Recall that, in this setup, the distinction between multi- and
univariate representation zeta functions is mute.

The case $d=1$ has been treated exhaustively by Rossmann. In
\cite[Theorem~A]{RossmannMA/17} he gives a fully explicit formula for the
submodule zeta function of any integral $\msfL_1$-representation in terms of
Dedekind zeta functions of number fields and combinatorial data. These two
types of ingredients reflect the Jordan decomposition of the endomorphism
representing the unique loop into a semi-simple and a nilpotent part.

For $d\geq 2$, an important class of examples arises from ideal zeta functions
of rings, defined as follows. Let $L$ be a commutative or anti-commutative
ring (for instance, a Lie ring), of finite additive rank. It is easy to see
that the \emph{ideal zeta function $\zeta^{\triangleleft}_{L}(s)$ of~$L$},
enumerating the (two-sided) ideals of finite index in $L$, is the zeta
function of the \emph{adjoint representation} $\ad (L) \subseteq \End_\Z(L)$
of $L$:
$$\zeta^{\triangleleft}_L(s) = \zeta_{\ad L \curvearrowright L}(s).$$
If $L$ is $d$-generated as a ring, then $\zeta^{\triangleleft}_L(s)$
is the zeta function associated with a representation~of
$\msfL_d$ and thus fits into the general framework developed in this
paper. The representation is nilpotent if and only if the ring $L$ is
nilpotent.

\begin{exm}[Heisenberg]\label{exa:heis}
  Consider the $\Z$-representation $V=(\Z^3,f_1,f_2)$ of $\msfL_2$ by $\Z^3$
  and the two endomorphisms
    $$f_1=\left(\begin{array}{ccc}0&0&0\\0&0&-1\\0&0&0\end{array}\right),
    \quad f_2=\left(\begin{array}{ccc}0&0&1\\0&0&0\\0&0&0\end{array}\right).$$
    One checks easily that, for any cDVR $\lri$, subrepresentations of
    $V(\lri)$ are in fact exactly the ideals of the Heisenberg Lie ring
    $\mfh(\lri)= \left( \begin{matrix} 0 & \lri & \lri \\ 0 & 0 & \lri \\ 0 &
        0 & 0\end{matrix}\right)$. In fact,
    \[\zeta_{V(\lri)}(s)=\zeta_{\mfh(\lri)}^{\triangleleft}(s) =
      \frac{1}{(1-q^{-s})(1-q^{1-s})(1-q^{2-3s})};\]
    cf.~\cite[Section~8]{GSS/88}. The endomorphisms $f_i$ are the linear maps
    $\ad x_i$, written with respect to the $\Z$-basis $(x_1,x_2,y)$ of
    $\mfh(\Z) = \la x_1,x_2,y \mid [x_1,x_2]=y, y \textup{ central}\ra_\Z$.

    The behaviour of the representation $V$ under base extension and
    restriction of scalars (cf.\ Remark~\ref{rem:base.ext}) is almost
    completely understood. Indeed, Schein and the second author computed the
    Euler factors $\zeta_{V(\Gri)\otimes_\Z \Zp}(s)$ in
    \eqref{equ:base.change} for all rational primes $p$ which are unramified
    in the ring of integers $\Gri$ in in \cite{SV1/15}; in \cite{SV2/15} they
    computed formulae for the non-split case. \cite[Conjecture~1.4]{SV1/15}
    would imply that the bad primes are exactly the ramified ones.
\end{exm}

It has long been known that nilpotent submodule zeta functions may or
may not satisfy the kind of local functional equation established in
Theorem~\ref{thm:main.refine}. (One of the smallest examples where it
fails is the filiform nilpotent Lie ring~$\Fil_4$; see
Section~\ref{subsec:exa.fil4}.) The homogeneity
condition~\ref{cond:quiver.hom} was first proposed by the second
author as a sufficient criterion for such functional equations:
\cite[Theorem~1.2]{VollIMRN/19} is the special case of
Theorem~\ref{thm:main.refine} for loop quivers.

\subsubsection{Graded submodule and graded ideal zeta functions}\label{subsec:graded}

Let $\mathcal{V}$ and $\Omega=\la c_1,\dots,c_d\ra$ be as in
Section~\ref{subsubsec:submod} and fix an $R$-module decomposition
$\mathcal{V} = \mathcal{V}_1 \oplus \dots \oplus \mathcal{V}_a$, not
necessarily compatible with~$\Omega$. The associated (\emph{graded
  submodule}) \emph{zeta function $\zeta^{\textup{gr}}_{\Omega
    \curvearrowright \mathcal{V}}(s)$ of $\Omega$ acting on
  $\mathcal{V}$} is the Dirichlet generating series enumerating
\emph{graded} (or \emph{homogeneous}) $\Omega$-invariant submodules
of~$\mathcal{V}$; cf.\ \cite[Remark~3.2]{RossmannTAMS/18}:
$$\zeta^{\textup{gr}}_{\Omega \curvearrowright \mathcal{V}}(s) =
\sum_{\mathcal{U}\leq_{\textup{gr}} \mathcal{V}} |\mathcal{V}:\mathcal{U}|^{-s}.$$

Graded submodule zeta functions, too, may be seen as zeta functions of
quiver representations. Indeed, let $\msfQ$ be the quiver with
vertices $Q_0=\{1,\dots,a\}$ and arrows $Q_1 = \{ \phi_{thk} \mid
(t,h)\in[a]^2, k\in[d]\}$, viz.\ $d$ arrows between any two vertices,
represented by modules $\mcL_{h}=\mathcal{V}_h$ for $h\in[a]$ and, for
$t\in[a]$ and $k\in[d]$, morphisms $f_{\phi_{thk}} =
\pi_h(c_k|_{\mathcal{V}_t})$, where $\pi_h$ denotes the projection
$\mathcal{V} \twoheadrightarrow \mathcal{V}_h$. Then
$\zeta^{\textup{gr}}_{\Omega \curvearrowright \mathcal{V}}(s)$ is the
univariate zeta function associated with this quiver representation
over~$R$; the multivariate one yields the obvious multivariate
refinement of~$\zeta^{\textup{gr}}_{\Omega \curvearrowright
  \mathcal{V}}(s)$.

Just as in the ungraded case discussed in Section~\ref{subsubsec:submod},
important examples arise from Dirichlet generating functions enumerating
certain ideals. Let $L$ be a ring as in Section~\ref{subsubsec:submod} with a
fixed $\Z$-module decomposition $L = L_1 \oplus \dots \oplus L_a$, not
necessarily compatible with the multiplication of~$L$. Then the \emph{graded
  ideal zeta function $\zeta^{\triangleleft_{\textup{gr}}}_{L}(s)$ of $L$ with
  respect to $L = \bigoplus_{h=1}^aL_h$}, enumerating the \emph{graded} ideals
of finite index in $L$ with respect to this decomposition, is the \emph{graded
  submodule representation of $\ad L$ acting on $L=\bigoplus_{h=1}^aL_h$},
i.e.\
$$\zeta^{\triangleleft_{\textup{gr}}}_{L}(s) = \zeta^{\textup{gr}}_{\ad L
  \curvearrowright L}(s).$$

Assume now that $L$ is a nilpotent Lie ring of finite additive rank~$n$ and
nilpotency class~$c$ with lower central series $(\gamma_i(L))_{i=1}^c$. For
$i\in[c]$, set $L_{i} = \gamma_i(L)/\gamma_{i+1}(L)$. The \emph{associated
  graded Lie ring} is $\grL = \bigoplus_{i=1}^c L_{i}$.  The \emph{graded
  ideal zeta function $\zeta^{\triangleleft_{\textup{gr}}}_L(s)$} of $L$ is
the graded ideal zeta function of $\grL$ with respect to the decomposition
$\grL = \bigoplus_{i=1}^c L_{i}$.

Theorem~\ref{thm:main.refine} implies that all the questions raised in
\cite[Question~10.2]{RossmannTAMS/18} have positive answers provided
that Condition~\ref{cond:quiver.hom} is satisfied. We collect further
consequences of Theorem~\ref{thm:main.refine} pertaining to (graded)
ideal zeta functions in Section~\ref{sec:graded.ideal}. The examples
explained in Section~\ref{subsec:exa.fil4} show, in particular, that
the problem of counting graded ideals may be homogeneous in the sense
of Condition~\ref{cond:quiver.hom} even when the problem of counting
\emph{all} ideals of finite index is not. In
Corollary~\ref{cor:funeq.free} we record, specifically, the generic
functional equations of the local graded ideal zeta functions
associated with the free nilpotent Lie ring $\mff_{c,d}$ on $d$
generators and of nilpotency class~$c$, for all $d$ and~$c$.

\begin{exm}[graded Heisenberg]
 Consider the $\Z$-representation $V=(\Z^2,\Z,f_1,f_2)$ with
 $\mcL_1=\la x_1,x_2\ra,\mcL_2=\la y \ra$, and the two endomorphisms
	$$f_1=\left(\begin{array}{ccc}0&0&0\\0&0&-1\\0&0&0\end{array}\right),
	\quad f_2=\left(\begin{array}{ccc}0&0&1\\0&0&0\\0&0&0\end{array}\right).$$
	One checks easily that, for any cDVR $\lri$, subrepresentations of
	$V(\lri)$ are in fact exactly the grade ideals of the Heisenberg Lie ring
	$\mfh(\lri)= \left( \begin{matrix} 0 & \lri & \lri \\ 0 & 0 & \lri \\ 0 &
		0 & 0\end{matrix}\right)$. In fact,
	\[\zeta_{V(\lri)}(\bfs)=\zeta_{\mfh(\lri)}^{\triangleleft_{\textup{gr}}}(\bfs) =
	\frac{1}{(1-q^{-s_1})(1-q^{1-s_1})(1-q^{-(2s_1+s_2)})}.\]
\end{exm}

\subsubsection{Quiver representation zeta functions as submodule zeta
  functions}\label{subsub:quiver.submod}

  While univariate quiver representation zeta functions afford new
  perspectives on (graded) submodule zeta functions, the latter
  actually comprise the former. Indeed, given a quiver $\msfQ$ and a
  ring~$R$, let $\pa$ be the path algebra of $\msfQ$ over~$R$,
  i.e.\ the $R$-algebra generated by the paths in~$\msfQ$ with
  multiplication induced by concatenation of paths. Then $\pa \cong N
  \oplus R^{|Q_0|}$, where $N$ is the ideal generated by paths of
  positive length. Given a representation $V$ of $\msfQ$ with
  underlying module $\bmcL$, the endomorphism algebra $\mcE = \mcE(V)$
  is the image of the natural map $N\rightarrow \End(\bmcL)$ induced
  by $V$. Clearly, $\mcE$ is nilpotent if and only if $V$ is.  If we
  map paths of length zero, viz.\ vertices $\iota\in Q_0$, to the
  projections $\pi_\iota:\bmcL\rarr\bmcL$ onto the direct summands
  $\mcL_\iota$ (followed by inclusion into~$\bmcL$), then the
  representation zeta function $\zeta_V(s)$ is the submodule zeta
  function associated with the image of (all of) $\pa$ within
  $\End(\bmcL)$. We thank Tobias Rossmann for pointing this out to us.

\subsubsection{Analytic properties}\label{subsubsec:ana}
Univariate local submodule zeta functions are known to be expressible
in terms of the $\mfp$-adic \emph{cone integrals} introduced in
\cite{duSG/00}; see \cite[Theorem~2.6(ii)]{Rossmann/15}. This implies,
in particular, that univariate analogues of Euler products such as
\eqref{equ:euler.prod} have rational abscissa of convergence and allow
for some meromorphic continuation; see \cite[Theorem~1.5(1)]{duSG/00}
for the case $R=\Z$ (the proof extends easily to general rings of
integers). A Tauberian theorem thus gives asymptotic estimates for the
partial sums $\sum_{i\leq m}a_i(V)$ in terms of the position and order
of the right-most pole of~$Z_V(s)$;
cf.\ \cite[Theorem~1.5(2)]{duSG/00}. It would be of interest to
explore extensions of these results to the multivariate setting. While
there are results on the domain of meromorphy of multivariate Euler
products (see, for instance, \cite{Delabarre/14}), applications to
counting functions akin to quiver representation zeta functions seem
to be thin on the ground; see \cite{Lins/18} for results on bivariate
representation and conjugacy class zeta functions.

\subsubsection{Further refinements}
The multivariate representation zeta function associated with a quiver
representation \eqref{def:zeta} could be further refined to take into
account (aspects of) the elementary divisor types of the lattices
$\Lambda_\iota$. For the ``abelian'' case ($|Q_0|=1$, $|Q_1|=0$;
cf.\ Example~\ref{exa:ab}) this is done in~\cite{Petrogradsky/07}; see
also \cite{CKK/17}. It would be of interest to determine to what
extent results such as Theorem~\ref{thm:main.refine} hold in this even
finer setup. The methodology of the current paper does not seem
appropriate for this task in general.

\subsection{Methodology and organization}
Our proof of Theorem~\ref{thm:main.refine} is based on generalizations of
techniques and results from \cite{Voll/10} and \cite{VollIMRN/19}:
Section~\ref{sec:new.blue} develops $\mfp$-adic machinery which we use
to prove the result in Section~\ref{sec:fun.eq.quiver.rep}.

In \cite{VollIMRN/19} the problem of enumerating submodules invariant
under ``homogeneous'' nilpotent algebras of endomorphisms was
approached using results of \cite{Voll/10}. To this end, the problem
was reformulated in terms of integer-valued ``weight functions'' on
the vertex sets of the affine Bruhat-Tits buildings associated with
groups of the form $\GL_n(K)$ for a local field~$K$, viz.\ homothety
classes of lattices inside~$K^n$. These weight functions were then
shown to be amenable to versions of a very general ``blueprint
result'' from \cite{Voll/10}, establishing functional equations for
certain $\mfp$-adic integrals.  This result has been used and
developed extensively to prove such functional equations for various
kinds of zeta functions associated with groups, rings, and modules;
see, for instance, \cite{AKOVI/13, StasinskiVoll/14, Rossmann_ask/18,
  KionkeKlopsch/19} for related work.

In Section~\ref{sec:fun.eq.quiver.rep} we rephrase the problem of enumerating
subrepresentations of homogeneous nilpotent quiver representations in terms of
weight functions on sets of \emph{tuples} of full $\mfp$-adic lattices,
generalizing those introduced in \cite{VollIMRN/19}. In
Section~\ref{sec:new.blue} we generalize the $\mfp$-adic blueprint result from
\cite{Voll/10} with a view towards these generalized weight functions.

As many of their precursors, the $\mfp$-adic (i.e.\ local) integrals in
Section~\ref{sec:new.blue} are obtained by localizing globally defined
data. Our results about these integrals typically only apply to all but
finitely many bad places, which explains the need to exclude finitely many
prime ideals in Theorem~\ref{thm:main.refine}. Corollary~\ref{cor:pos.char} follows
immediately from the formulae for the generic (``good'') places by invoking
the transfer principle from model theory, which justifies the ``transfer''
between formulae for $\mfp$-adic integrals in characteristic zero and their
analogues in positive characteristic, provided the residue field
characteristic is sufficiently large.

In Sections~\ref{subsec:inf.overview.blue} and \ref{subsec:inf.overview.funeq}
we provide informal overviews of the two sections which, taken together, form
the technical core of this paper.

In Sections~\ref{sec:graded.ideal} and~\ref{sec:further.exa} we discuss our
main result Theorem~\ref{thm:main.refine} in some special contexts, viz.\ (graded)
ideal zeta functions of nilpotent Lie rings and certain specific quiver
representations which do not come from this classical setup.  While some of
the examples we develop in Section~\ref{sec:further.exa} are of a
combinatorial nature, others give an inkling of the subtlety of the variation
of the Euler factors of a global quiver representation zeta function with the
place: general formulae involve, in an essential way, the numbers of rational
points of algebraic varieties over finite fields. This notwithstanding, both
sections may be read independently of Sections~\ref{sec:new.blue}
and~\ref{sec:fun.eq.quiver.rep}.

\subsection{Notation}\label{subsec:not}

By $R$ we denote a ring. Usually, it will either be the ring of integers
$\Gri$ of a global field or a compact discrete valuation ring (cDVR), either
of characteristic zero (such as the completion $\Gri_{\mfp}$ of $\Gri$ at a
non-zero prime ideal $\mfp$ of $\Gri$) or of positive characteristic (such as
the ring of formal power series $\Fq\llbracket T \rrbracket$ over a finite
field~$\Fq$). We write $\mfp$ for the maximal ideal of a cDVR $\lri$ and $q$
resp.\ $p$ for its residue field's cardinality resp.\ characteristic. By $v$
or $v_{\mfp}$ we denote the (normalized) $\mfp$-adic valuation on $\lri$, but
also, by extension, on vectors and matrices over $\lri$: if
$\bfx=(x_1,\dots,x_a)\in\lri^a$, then
$v(\bfx)=\min\{v(x_i)\mid i\in\{1,\dots,a\}\}$. Occasionally we refer to a
uniformizer $\pi$ of $\lri$, viz.\ an element~$\pi\in\mfp\setminus\mfp^2$. In
general, however, the notation $\mfp^m$ refers to the Cartesian product
$\mfp\times\dots\times\mfp$ with $m$ factors. We trust that the respective
contexts will prevent misunderstandings. By $K$ we denote a field, usually a
global or local field such as the field of fractions of~$R$.

Given matrices $A$ and $B$ over $\lri$ with the same number of columns, we
write $A\leq B$ if each row of $A$ is contained in the $\lri$-row span of~$B$.

Throughout, $\msfQ$ will be a quiver, with vertices $Q_0$ and arrows
$Q_1$. Often we will write $a=|Q_0|$ and $b=|Q_1|$. Special classes of quivers
discussed include, among others, loop quivers~$\msfL_d$, Kronecker
quivers~$\msfK_a$, star quivers~$\msfS_a$ and their duals~$\msfS^*_a$. By $V$
we will denote a representation of a quiver over a ring~$R$.

We denote by $\N=\{1,2,\dots\}$ the set of natural numbers and set
$X_0 = X \cup \{0\}$ for a subset $X\subset \N$. Given $n\in\N_0$, we write
$[n]=\{1,\dots,n\}$; for $m,n\in\N_0$ we write
$]m,n]=\{m+1,m+2,\dots,n\}$. The power set of a set $X$ is denoted
$\mcP(X)$. We write $I=\{i_1,\dots,i_\ell\}_<\subset \N_0$ to stress that
$i_1<\dots < i_\ell$. We write $t=q^{-s}$, where $s$ is a complex variable. Given numbers $x_1,x_2,\dots$ and multiplicities $e_1,e_2\in\N_0$ we set
$$\left( x_1^{(e_1)},x_2^{(e_2)},\dots\right) = \left(\underbrace{x_1,\dots,x_1}_{\textup{$e_1$ times}}, \underbrace{x_2,\dots,x_2}_{\textup{$e_2$ times}},\dots\right).$$

Given a property $\phi$, the ``Kronecker delta'' $\delta_\phi$ is equal to 1
if $\phi$ holds and equal to 0 otherwise.

\section{Graded ideal zeta functions of nilpotent Lie
  rings}\label{sec:graded.ideal}

In this section we discuss our main result in the light of some
special classes of graded ideal zeta functions of nilpotent Lie
rings. For simplicity and to ease comparison with thexisting
literature on these zeta functions we restrict ourselves to univariate
zeta functions.

\subsection{$\Fil_4$ vs.\ $M_4$}\label{subsec:exa.fil4}
Consider the class-$4$-nilpotent Lie ring
  \[\Fil_4:=\langle x_1,x_2,x_3,x_4,x_5\mid [x_1,x_2]=x_3,[x_1,x_3]=x_4,[x_1,x_4]=x_5,[x_2,x_3]=x_5\rangle.\]
  (By convention, all commutator relations other than those following from the
  given ones are assumed to be trivial.) It is known that its local ideal zeta
  functions $\zeta^{\triangleleft}_{\Fil_4(\Zp)}(s)$ do not satisfy a
  functional equation of the form described in Theorem~\ref{thm:main.refine}; see
  \cite[Theorem~2.39]{duSWoodward/08} and
  \cite[Example~4.1]{VollIMRN/19}. Informally speaking, the homogeneity
  condition is violated by the ``inhomogeneity'' of the map $\ad x_2$. Recall
  from Section~\ref{subsubsec:submod} that this is an example of a zeta function
  of a representation of~$\msfL_2$: 
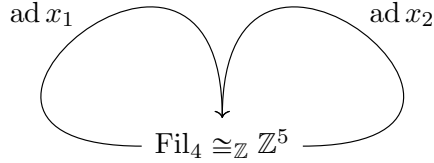
\vspace*{-1.6cm}
  	\[\begin{tikzcd}
            \Fil_4 \cong_\Z\Z^{5}\arrow[out=0,in=90,loop,swap,"\ad x_2"]
            \arrow[out=180,in=90,loop,"\ad x_1"]
	\end{tikzcd}\]
	
        Consider now the associated graded Lie ring 
\[ \gr \Fil_4 = \langle {x_1},{x_2}\rangle \oplus\langle  {x_3}\rangle\oplus\langle x_4\rangle\oplus\langle x_5\rangle=:\mcL_1\oplus\mcL_2\oplus\mcL_3\oplus\mcL_4.\]
(By slight abuse of notation we continue to write here
$x_i\in\gamma_j(\Fil_4)$ for its image in
$\gamma_j(\Fil_4)/\gamma_{j+1}(\Fil_4)$.)  The graded ideal zeta
function $\zeta^{\triangleleft_{\textup{gr}}}_{\Fil_4}(s)$ is the zeta
function of the following quiver representation:
	
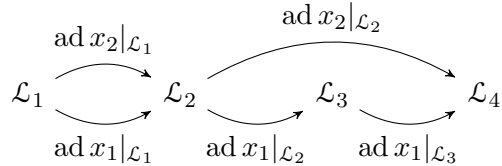
\[
		\begin{tikzpicture}
	[->,>=stealth',shorten >=1pt,auto,node distance=2.0cm,
	main node/.style={}]
	
	\node[main node] (a) {$\mcL_1$};
	\node[main node] (b) [right of=a] {$\mcL_2$};
	\node[main node] (c) [right of=b] {$\mcL_3$};
	\node[main node] (d) [right of=c] {$\mcL_4$};
	
	\path
	(a) edge [bend left]  node {$\ad x_2|_{\mcL_1}$} (b)
	edge [bend right,swap]  node{$\ad x_1|_{\mcL_1}$}(b)
	(b) 
	edge [bend left] node  {$\ad x_2|_{\mcL_2}$}  (d)
	edge [bend right,swap] node {$\ad x_1|_{\mcL_2}$}(c)
	(c) 
	
	edge [bend right,swap] node {$\ad x_1|_{\mcL_3}$}(d);
	\end{tikzpicture}
	\]
        However, the ``inhomogeneous arrow'' $\ad x_2|_{\mcL_2}$ is redundant,
        as obviously $\ad x_2|_{\mcL_2} = (\ad x_1)^2|_{\mcL_2}$.  Omitting it
        yields the quiver representation modelling the graded ideal zeta
        function of the maximal class Lie ring
\begin{equation}\label{def:M4}
  M_4:=\langle x_1,x_2,x_3,x_4,x_5\mid [x_1,x_2]=x_3,[x_1,x_3]=x_4,[x_1,x_4]=x_5\rangle.
\end{equation}
The latter is easily seen to satisfy the conditions of
Theorem~\ref{thm:main.refine}. For the simple explicit formula of the graded ideal
zeta function, see \cite[Proposition~3.5]{RossmannTAMS/18}, where $M_4$ goes
by the name~$\mathfrak{m}(4)$.

\subsection{Graded ideal zeta functions of free nilpotent Lie rings}
In \cite{LeeVoll/18} we investigated graded ideal zeta functions
$\zeta^{\idealgr}_{\mff_{c,d}(R)}(s)$ associated with the free
nilpotent Lie rings $\mff_{c,d}$ on (Lie) generators $x_1,\dots,x_d$,
of nilpotency class~$c$, for various rings~$R$. By the above, these
zeta functions may be interpreted as zeta functions of integral
(adjoint) representations of quivers $\msfF_{c,d}$ on $c$ vertices
$v_1,\dots,v_c$, with exactly $d$ arrows between $v_i$ and $v_{i+1}$
for all $i\in[c-1]$. In our application, as shown in
\cite{LeeVoll/18}, the vertices $v_i$ ($i\in[c]$) are represented by
free $R$-modules of ranks $W_d(i) = \frac{1}{i}\sum_{j|i}\mu(j)
d^{i/j}$,
viz.\ $\mathcal{L}_{i}:=\gamma_i(\mff_{c,d})/\gamma_{i+1}(\mff_{c,d})\otimes_\Z
R$, where $\mu$ denotes the M\"{o}bius function ; the arrows are
represented by the maps $\ad x_k|_{\mathcal{L}_{i}}: {\mathcal{L}_{i}}
\rarr {\mathcal{L}_{i+1}}$, where~$k\in[d]$.

In \cite[Theorem~1.1]{LeeVoll/18} we recorded a formula for the graded ideal
zeta functions~$\zeta^{\idealgr}_{\mff_{3,3}(\lri)}$, valid for all
cDVRs~$\lri$. The relevant representations are all of the quiver $\msfF_{3,3}$
(see Figure~\ref{fig:L33}),
\begin{figure}
 \centering
 \caption{The quiver $\msfF_{3,3}$}
 \label{fig:L33}
$$	\begin{tikzcd}[arrow style=tikz,>=stealth,row sep=4em]
	\bullet \arrow[rr,shift left=.0ex]
	\arrow[rr,shift left=.8ex]
	\arrow[rr,shift right=.8ex]
	&&\bullet\arrow[rr,shift left=.ex]
	\arrow[rr,shift left=0.8ex]
	\arrow[rr,shift right=.8ex]
	&&\bullet
	\end{tikzcd}$$
\end{figure}
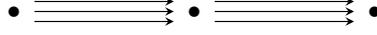
with rank vectors $(W_3(1),W_3(2), W_3(3)) = (3,3,8)$. The paper also
contains explicit formulae for all $c\leq 2$ and
$(c,d)\in\{(3,3),(3,2),(4,2)\}$. For larger values of $c$ and~$d$,
explicit computations seem currently out of reach.

We made several general conjectures about graded ideal zeta functions
associated with free nilpotent Lie rings
in~\cite[Section~6]{LeeVoll/18}. The one pertaining to local
functional equations is implied by Theorem~\ref{thm:main.refine}.

\begin{cor}\cite[Conjecture~6.2]{LeeVoll/18}\label{cor:funeq.free} 
  For almost all primes $p$ and all cDVRs $\lri$ of residue field
  cardinality~$q$ and residue field characteristic~$p$,
\[
\zeta_{\mff_{c,d}(\lri)}^{\idealgr}(s)|_{q\rarr
  q^{-1}}=(-1)^{r}q^{\sum_{i=1}^{c}\left(\binom{W_d(i)}{2}-(c+1-i)W_d(i)s\right)}\zeta_{\mff_{c,d}(\lri)}^{\idealgr}(s).
\]
\end{cor}
Indeed, the adjoint representation is clearly homogeneous in the sense of
Condition~\ref{cond:quiver.hom}. Corollary~\ref{cor:funeq.free} is a graded
analogue of \cite[Theorem~4.4]{VollIMRN/19}.

\subsection{Some Lie rings of maximal class and their
  amalgams}
Given an (integer) partition
$\lambda=(\lambda_{1},\ldots,\lambda_{r})\in\N^{r}$, with
$\lambda_{1}\geq\lambda_{2}\geq\cdots\geq\lambda_{r}$, consider the
class-$\lambda_{1}$-nilpotent Lie ring
\[\mcL_{\lambda}=\langle x_{0},\{x_{i,j}\}_{i\in[r],j\in[\lambda_{i}]}\mid\forall i\in[r],j\in[\lambda_{i}-1]:[x_{0},x_{i,j}]=x_{i,j+1}\rangle_{\Z},\]
on $1+r$ Lie generators and of $\Z$-rank~$1+\sum_{i=1}^r\lambda_i$; see also
\cite[Section~4.3]{VollIMRN/19}. For $r=1$ and $\lambda_{1}\geq2$ this yields
the Lie ring $M_{\lambda_{1}}$ of maximal class $\lambda_{1}$ described in
\cite[p.~99]{duSWoodward/08}; for $M_4$, see~\eqref{def:M4}. In general,
$\mcL_{\lambda}$ is obtained by amalgamating several such Lie rings
along~$x_{0}$.  Consider the special case $\lambda=c^{(r)}=(c,\ldots,c)$ of
\emph{rectangular} partitions.  In \cite[Chapter~4]{Lee/19thesis}, the first
author investigated graded ideal zeta functions
$\zeta_{\mcL_{c^{(r)}}(\mfo)}^{\idealgr}(s)$ associated with
$\mcL_{c^{(r)}}$. He derived explicit (if involved) combinatorial formulae for
these rational functions, valid for all residue characteristics, spelling out
the cases $(c,r)\in\{(3,2),(3,3),(3,4),(4,2),(5,2)\}$ in detail. By the above,
these zeta functions may be interpreted as zeta functions of integral
(adjoint) representations of quivers on $c$ vertices $v_1,\dots,v_c$,
represented by free $R$-modules of ranks $r+\delta_{i,1}$, with exactly
$1 +r\delta_{i,1}$ arrows between $v_i$ and $v_{i+1}$ for all~$i\in[c-1]$. In
\cite[Theorem~4.2.12]{Lee/19thesis}, he established general local functional
equations for $c=3$. In fact, his combinatorial approach easily extends to
cover \emph{near rectangular} partitions $\lambda=(c^{(r_1)},1^{(r_2)})$,
where $r_1,r_2\in\N_0$. Theorem~\ref{thm:main.refine} yields the following.
\begin{cor}\label{cor:funeq.lcr}
  Let $c\in\N$ and $r_1,r_2\in\N_0$. For all cDVRs $\lri$ of residue
  field cardinality~$q$
	\begin{multline*}
          \zeta_{\mcL_{(c^{(r_1)},1^{(r_2)})}(\lri)}^{\idealgr}(s)|_{q\rarr
            q^{-1}} =\\
          (-1)^{1+cr_1+r_2}q^{\left(\binom{1+r_1+r_2}{2}+(c-1)\binom{r_1}{2}\right)-\left(c+\binom{c+1}{2}r_1+r_2\right)s}\zeta_{\mcL_{(c^{(r_1)},1^{(r_2)})}(\lri)}^{\idealgr}(s).
	\end{multline*}
\end{cor}
Indeed, as in \cite[Lemma~4.7]{VollIMRN/19} one checks that the adjoint
representation is homogeneous in the sense of Condition~\ref{cond:quiver.hom}
if and only if $\lambda$ is a near rectangle. Graded ideal zeta functions of
Lie rings arising from non-near rectangular partitions thus do not fall in the
remit of the current paper's methodology. Corollary~\ref{cor:funeq.lcr} is a
graded analogue of \cite[Theorem~4.8]{VollIMRN/19}.  As in the ``ungraded''
case treated in \cite[Question~4.9]{VollIMRN/19} it is natural to ask whether
the ``near rectangle''-condition is necessary for local functional equations
for the \emph{graded} ideal zeta
functions~$\zeta^{\idealgr}_{\mcL_\lambda(\lri)}(s)$.

\subsection{Further examples}
Univariate quiver representation zeta functions being submodule zeta
functions (cf.\ Section~\ref{subsub:quiver.submod}) makes them
amenable to the functionality of Rossmann's computer algebra
package~$\mathsf{Zeta}$~(\cite{Zeta}). In
\cite[Section~10]{RossmannTAMS/18}, Rossmann reports numerous
computations of graded ideal zeta functions obtained
with~$\mathsf{Zeta}$. It seems noteworthy that the formulae listed in
\cite[Table 2]{RossmannTAMS/18} satisfy the relevant functional
equation if and only if (!) they satisfy Condition
\ref{cond:quiver.hom} with respect to the bases suggested in
\cite[Section~5.1]{Kuzmich/99}. These observations support the
speculation that homogeneity may also be a necessary condition for
local functional equations. In Section~\ref{subsec:P-part} we discuss
a class of nilpotent quiver representations where this is provably the
case, at least in the univariate setup; see
Theorem~\ref{thm:stanley.funeq}.

\section{Integral nilpotent quiver representations beyond nilpotent Lie
  rings}\label{sec:further.exa}

In this section we discuss examples of univariate quiver
representation zeta functions not arising from the ``classical''
context of (graded) ideal zeta functions of nilpotent Lie
rings. Whenever these examples exhibit functional equations of the
form~\eqref{equ:funeq.refine.univariate}, they are explained by
Theorem~\ref{thm:main.refine}. Section~\ref{subsec:P-part} explores
connections between integral thin representations of Hasse quivers of
posets and $\ps$-partitions, a combinatorial concept. Representations
of star quivers star in Section~\ref{subsec:star.quivers}, their duals
in Section~\ref{subsec:dual.star.quivers}; representations of
Kronecker quivers are the subject of
Section~\ref{subsec:kronecker}. Throughout, $\lri$ denotes a cDVR of
arbitrary characteristic, with maximal ideal~$\mfp$ and residue field
cardinality~$q$.

\subsection{Integral thin representations of Hasse quivers and $\ps$-partitions}\label{subsec:P-part}

The exposition in this section leans closely on
\cite[Sec.~13.5]{Stanley/12}, to which we refer for further
details. Let $\ps$ be a partially ordered set (or poset) of
cardinality $n$. Without loss of generality $\ps$ is a natural partial
order on~$[n]$, i.e.\ if $i <_\ps j$ then $i<j$ for all
$i,j\in\ps$. Recall that a \emph{$\ps$-partition of $m\in\N_0$} is an
order-reversing (!)\ map $\sigma: \ps\rightarrow\N_0$ satisfying
$|\sigma|:=\sum_{i\in\ps}\sigma(i)=m$. (The labelling $\omega$
referred to in \cite{Stanley/12} is subsumed by our identification of
$\ps$ with $[n]$.) We write $a_{m,\ps}:=\#\{\sigma\mid |\sigma| = m\}$
for the number of $\ps$-partitions of $m$. As in
\cite[(3.62)]{Stanley/12} we set
$$G_{\ps}(X) := \sum_{m=0}^\infty a_{m,\ps} X^m,$$ for a
variable~$X$. The following result is
\cite[Theorem~3.15.7]{Stanley/12}. It establishes that this generating
function is, in fact, rational and expresses it in terms of the major
index $\maj$ on~$S_n$ (see \cite[Sec.~1.4]{Stanley/12}), restricted to
the subset $\mathbb{L}(\ps)\subseteq S_n$ of linear extensions
of~$\ps$.

\begin{thm}[R.\ Stanley]\label{thm:stanley}
  $$G_\ps(X) = \frac{\sum_{\pi\in
      \mathbb{L}(\ps)}X^{\maj(\pi)}}{\prod_{i=1}^n(1-X^i)}.$$
\end{thm}

Recall further that $\ps$ satisfies the \emph{$\delta$-chain
  condition} if, for all $x\in\ps$, all maximal chains in the
principal dual order ideals $\{x'\in\ps \mid x' \geq_\ps x\}$ have the
same length. For $x\in\ps$, let $\delta(x)$ denote the length of a
longest chain in $\ps$ starting at~$x$. Set further
$\delta(\ps)=\sum_{x\in \ps}\delta(x)$. The following result is
\cite[Theorem~3.15.16]{Stanley/12}.

\begin{thm}[R.\ Stanley]\label{thm:stanley.funeq}
  The poset $\ps$ satisfies the $\delta$-chain condition if and only if the following functional equation holds:
  $$G_{\ps}(X^{-1}) = (-1)^n X^{\delta(\ps)} G_\ps(X).$$ 
\end{thm}

In fact, $\ps$-partitions may be viewed as integral thin quiver
representations of Hasse quivers of posets. Indeed, let $Q_{\ps}$ be the Hasse
quiver of $\ps$. It has vertices $Q_{\ps,0}=\ps$ and arrows
$Q_{\ps,1}=\{(i,j)\mid i \lessdot_{\ps} j\}$, where $i \lessdot_{\ps} j$ means
that $j$ covers $i$ in $\ps$, i.e.\ $i <_\ps j$ and
$\nexists k\in \ps:\: i <_\ps k <_\ps j$. Consider the thin $\Z$-representation
$V_\ps$ of $Q_\ps$ where all arrows are represented by identity maps. Then,
for every cDVR $\lri$ with maximal ideal~$\mfp$, we find that
\begin{equation}\label{eq:stanley=thin}
  \zeta_{V_\ps(\lri)}(s) =
  \sum_{\substack{(e_1,\dots,e_n)\in\N_0^n\\ e_j \leq e_i \textup{ if
      } i \leq_\ps j}} \prod_{k=1}^n |\lri:\mfp^{e_k}|^{-s} =
  \sum_{m=0}^\infty a_{m,\ps} t^m = G_\ps(t).
  \end{equation}

It is easy to check that the endomorphism algebra $\mcE(V_\ps)$ is
homogeneous in the sense of Condition~\ref{cond:quiver.hom} if and
only if $\ps$ satisfies the $\delta$-chain condition. Our
Theorem~\ref{thm:main.refine} thus implies (even a multivariate
version of) the ``only if''-part of
Theorem~\ref{thm:stanley.funeq}. We find it remarkable that, at least
in this combinatorialy tightly costrained setup, the homogeneity
condition is also necessary for a generic local functional
equation. It is of great interest to determine the precise extension
of this phenomenon, both in the multi- and univariate setup.

\subsection{Star quivers}\label{subsec:star.quivers}
Let $a\in\N$ and consider the quiver $\msfS_{a}$ consisting of $a$
vertices $v_1,\dots,v_a$ and $a-1$ arrows, all pointing away from the
central vertex $v_1$, describing a ``star with $a-1$ rays''. In the
terminology of Section~\ref{subsec:P-part}, $\msfS_{a}$ is the Hasse
quiver of the poset obtained from an antichain on $a-1$ vertices,
augmented by a minimal element $\widehat{0}$.

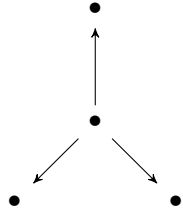
\begin{figure}
 \centering
 \caption{The star quiver $\msfS_4$}
 \label{fig:Q3}
 \begin{tikzpicture}[->,>=stealth',shorten >=1pt,auto,node distance=1.5cm,
main node/.style={}]

\node[main node] (a) {$\bullet$};
\node[main node] (b) [below left of=a] {$\bullet$};
\node[main node] (c) [below right of=a] {$\bullet$};
\node[main node] (d) [above of=a] {$\bullet$};  

\path
(a) edge node {} (d)
(a) 
(a) 
edge node  {} (b)
(a) 
edge node   {} (c);
\end{tikzpicture}
\end{figure}
We consider $\lri$-representations $V_{m,a}(\lri)$ of $\msfS_{a}$, where every
vertex is represented by $\lri^m$ and every arrow by the identity map.

\subsubsection{$m=1$}
In this case, we obtain a family of thin $\lri$-representations
of~$\msfS_{a}$. Recall, say from \cite{Carlitz/54,Carlitz/75}, \emph{Carlitz'
  $q$-Eulerian polynomial}
\[C_{a-1}(x,q) = \sum_{w\in S_{a-1}} x^{\des(w)} q^{\maj(w)} \in \Z[x,q],\]
where $\des$ is the \emph{descent statistic} and $\maj$ the \emph{major index}
on the symmetric group $S_{a-1}$; see, for instance,
\cite[Sec.~1.4]{Stanley/12}.

\begin{pro}
$$\zeta_{V_{1,a}(\lri)}(s) =
  \frac{C_{a-1}(t,t)}{\prod_{i=1}^{a}(1-t^i)}.$$
\end{pro}

\begin{proof} We have
	\begin{equation}\label{equ:macmahon}
		\zeta_{V_{1,a}(\lri)}(s) = \sum_{r=0}^\infty t^r \left(
		\frac{1-t^{r+1}}{1-t} \right)^{a-1} =
		\frac{C_{a-1}(t,t)}{\prod_{i=1}^{a}(1-t^{i})},
	\end{equation}
        where the first equality follows from the definition of
        $\zeta_{V_{1,a}(\lri)}(s)$ and MacMahon's (\cite[\textsection
          462, Vol.~2, Ch.~IV, Sect.~IX]{MacMahon/16} (see also
        \cite[(2.3)]{CV/18}) implies the second one.

        Alternatively, the statement follows easily by combining
        Theorem~\ref{thm:stanley} and~\eqref{eq:stanley=thin}.
\end{proof}
	
\begin{rem}\
 \begin{enumerate}
  \item Note that $\zeta_{V_{1,a}(\lri)}(s)$ has a pole of order $a$ at
    $t=1$ and $\left.(1-t)^{a}\zeta_{V_{1,a}(\lri)}(s)\right|_{t=1} =
    \frac{1}{a}$.
  \item The coefficients $n_{a,i}$ of the generating function
    $\zeta_{V_{1,a}(\lri)}(s) = \sum_{i=0}^\infty n_{a,i}t^i$ are the numbers
    of compositions of $i$ into $a$ parts whose first part is maximal.
\end{enumerate}
\end{rem}

\begin{exm} The following formulae were obtained using Maple\footnote{Maple is
    a trademark of Waterloo Maple Inc.}.
  \begin{enumerate}
		\item $\zeta_{V_{1,1}(\lri)}(s) = \frac{1}{1-t}$
		\item $\zeta_{V_{1,2}(\lri)}(s) = \frac{1}{(1-t)(1-t^2)}$
		\item $\zeta_{V_{1,3}(\lri)}(s) = \frac{1+t^2}{(1-t)(1-t^2)(1-t^3)}$
		\item $\zeta_{V_{1,4}(\lri)}(s) = \frac{1+2t^2 + 2t^3 + t^5}{(1-t)(1-t^2)(1-t^3)(1-t^4)}$ 
		\item $\zeta_{V_{1,5}(\lri)}(s) = \frac{1 + 3t^2 + 5t^3 + 3t^4 + 3t^5 + 5t^6 + 3t^7 +
			t^9}{(1-t)(1-t^2)(1-t^3)(1-t^4)(1-t^5)}$
		\item $\zeta_{V_{1,6}(\lri)}(s) = \frac{1+4t^2+9t^3+9t^4+10t^5+16t^6+22t^7+16t^8+10t^9+9t^{10}+9t^{11}+4t^{12}+t^{14}}{(1-t)(1-t^2)(1-t^3)(1-t^4)(1-t^5)(1-t^6)}$
		\item$\zeta_{V_{1,7}(\lri)}(s) = \frac{N_{1,7}(q,t)}{\prod_{i=1}^{7}(1-t^i)},$
  \end{enumerate}
  where
  \begin{multline*}
    N_{1,7}(q,t) =
    1+5t^2+14t^3+19t^4+24t^5+40t^6+66t^7+80t^8+76t^9+70t^{10}+\\76t^{11}+80t^{12}+66t^{13}+40t^{14}+24t^{15}+19t^{16}+14t^{17}+5t^{18}+t^{20}.\end{multline*}
\end{exm}

\subsubsection{$m=2$}\label{subsubsec:dual.star.n=2}

Given an integer partition $\lambda=(\lambda_1,\lambda_2)$, let
$\zeta_{\lambda,\lri}(s)$ be the Dirichlet polynomial enumerating the
$\lri$-submodules of the finite $\lri$-module $\lri/\mfp^{\lambda_1}
\times \lri/\mfp^{\lambda_2}$; if $\lri=\Zp$, this is just the
``subgroup zeta function'' enumerating the subgroups of the finite
abelian $p$-group $C_{p^{\lambda_1}} \times C_{p^{\lambda_2}}$. It is
easy to write down an explicit formula for
$\zeta_{\lambda,\lri}(s)$---in principle also for a general
partition---using, e.g., a formula due to Birkhoff; see, for example,
\cite[Section~2.5]{SV1/15}. It is obviously a polynomial in $q$ and
$t=q^{-s}$. Given a lattice $\Lambda\leq \lri^2$ of finite index, we
write $\lambda(\lri^2/\Lambda)=\lambda$ if $\lri^2/\Lambda$ has
type~$\lambda$.

\begin{pro}
\begin{equation}\label{equ:V2a}
  \zeta_{V_{2,a}(\lri)}(s) = \sum_{r_0 = 0}^\infty 
  t^{2r_0}\left(\left(\zeta_{(r_0,r_0),\lri}(s)\right)^{a-1} +
  (1+q^{-1})\sum_{r_1=1}^\infty(qt)^{r_1} \left(
  \zeta_{(r_0+r_1,r_0),\lri}(s)\right)^{a-1} \right).
\end{equation}
\end{pro}

\begin{proof} Clearly
  $$\zeta_{V_{2,a}(\lri)}(s) = \sum_{\substack{\Lambda_i \leq
      \lri^2,\, i\in[a],\\ \Lambda_1 \leq \Lambda_j,\, j\in \{2,\dots,a\}
    }}\prod_{i\in[a]}|\lri^2:\Lambda_i|^{-s} =
  \sum_{\Lambda_1\leq\lri^2}|\lri^2:\Lambda_1|^{-s}\left(\zeta_{\lambda(\lri^2/\Lambda_1),\lri}(s)\right)^{a-1}.$$
  The statement now follows from the well-known formula ($r_0,r_1\in\N_0$)
  $$\#\left\{\Lambda\leq \lri^2 \mid
    \lambda(\lri^2/\Lambda)=(r_0+r_1,r_0)\right\}
  = \begin{cases} 1, & \textup{ if } r_1=0,\\
    (1+q^{-1})q^{r_1}, & \textup{ if }r_1 > 0.\end{cases}\qedhere$$
\end{proof}

\begin{exm}The following formulae were obtained using Maple.
	\begin{enumerate}
		\item $\zeta_{V_{2,1}(\lri)}(s) =
                  \frac{1}{(1-t)(1-qt)}$
		\item $\zeta_{V_{2,2}(\lri)}(s) = \frac{1}{(1-t)(1-t^2)(1-qt)(1-qt^2)}$
		\item $\zeta_{V_{2,3}(\lri)}(s) =
		\frac{(1+t^2)(1-qt^4)}{(1-t)(1-t^2)(1-t^3)(1-qt)(1-qt^2)^2(1-qt^3)}$
                		\item $\zeta_{V_{2,4}(\lri)}(s) = \frac{N_{2,4}(q,t)}{(1-t)^2(1-t^3)(1-t^4)(1-qt)(1-qt^2)^2(1-qt^3)^2(1-qt^4)(1-q^3t^5)}$,

	\end{enumerate}
	where
	\begin{multline*}
	N_{2,4}(q,t) =1-t+3t^2+qt^2-t^3+q^2t^4-qt^4+t^4+qt^5-5qt^6+qt^7-3q^2t^7\\
	-2q^3t^7-5qt^8-2q^3t^9+q^2t^9+2qt^9-2q^4t^{10}-q^3t^{10}+2q^2t^{10}\\
	+5q^4t^{11}-q^4t^{12}+3q^3t^{12}+2q^2t^{12}+5q^4t^{13}-q^4t^{14}\\
	-q^5t^{15}+q^4t^{15}-q^3t^{15}+q^5t^{16}-q^4t^{17}-3q^5t^{17}+q^5t^{18}-q^5t^{19}.
	\end{multline*}
\end{exm}

\begin{rem}
  Viewing \eqref{equ:V2a} as a ($q$-analogue of a) generalization of
  MacMahon's formula~\eqref{equ:macmahon}, it remains a challenge to compute
  and interpret combinatorially formulae for local zeta functions of star
  quiver representations of the form~$V_{m,a}(\lri)$.
\end{rem}

\subsection{Dual star quivers}\label{subsec:dual.star.quivers}

For $a\in\N$ let $\msfS_{a}^*$ be the dual of the quiver~$\msfS_{a}$,
consisting of $a$ vertices $v_1,\dots,v_a$ and $a-1$ arrows, all
pointing towards the central vertex~$v_1$. In the terminology of
Section~\ref{subsec:P-part}, $\msfS_{a}^*$ is the Hasse quiver of the
poset obtained from an antichain on $a-1$ vertices, augmented by a
maximal element~$\widehat{1}$.

\begin{figure}
 \centering
 \caption{The dual star quiver $\msfS_4^*$}
 \label{fig:Q3.star}
\begin{tikzpicture}[->,>=stealth',shorten >=1pt,auto,node distance=1.5cm,
	main node/.style={}]
	
	\node[main node] (a) {$\bullet$};
	\node[main node] (b) [below left of=a] {$\bullet$};
	\node[main node] (c) [below right of=a] {$\bullet$};
	\node[main node] (d) [above of=a] {$\bullet$};  
	
	\path
	(d) edge node {} (a)
	(a) 
	(b) 
	edge node  {} (a)
	(c) 
	edge node   {} (a);
	\end{tikzpicture}
\end{figure}
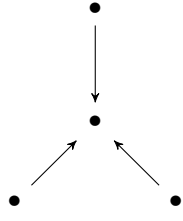

\subsubsection{} We first consider dual representations $V_{m,a}^*(\lri)$ of
the representations $V_{m,a}(\lri)$ introduced in
Section~\ref{subsec:star.quivers}. Recall that every vertex is represented by
$\lri^m$ and every arrow by the identity map. It turns out that---in contrast
to the zeta functions $\zeta_{V_{m,a}(\lri)}(s)$ discussed in
Section~\ref{subsec:star.quivers}---the associated zeta functions have a
rather simple form.

\begin{pro}\label{prop:dual.star} 
$$\zeta_{V^*_{m,a}(\lri)}(s) =  \zeta_{\lri^m}(as)\,\zeta_{\lri^m}(s)^{a-1}.$$
\end{pro}

\begin{proof}
This follows immediately from the observation that
\begin{multline*}
  \zeta_{V^*_{m,a}(\lri)}(s) =\\ \sum_{\substack{\Lambda_i\leq \lri^m,\,
      i\in[a],\\ \Lambda_j \leq \Lambda_1,\ j\in\{2,\dots,a\}}}
  \prod_{i\in[a]}|\lri^m:\Lambda_i|^{-s} = \sum_{\Lambda_1\leq \lri^m}
  |\lri^m:\Lambda_1|^{-sa}\prod_{j=2}^{a} \sum_{\Lambda_j'\leq
    \Lambda_1}|\Lambda_1:\Lambda_j'|^{-s}.\qedhere
  \end{multline*}
\end{proof}

\subsubsection{} The following example hints at the general fact that, even in small
examples, geometric as well as arithmetic considerations need to
complement the combinatorial arguments we have encountered so
far. This phenomenon will also feature in
Section~\ref{subsec:kronecker} and foreshadows the general situation
set out in Sections~\ref{sec:new.blue}
and~\ref{sec:fun.eq.quiver.rep}.

Consider the $\Z$-representation $V$
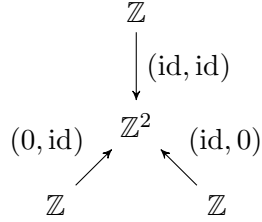
\begin{figure}%
  [htb!]
 \centering
 \caption{An integral representation of the dual star quiver~$\msfS_4^*$}
 \label{fig:D4}
	\begin{tikzpicture}[->,>=stealth',shorten >=1pt,auto,node distance=1.5cm,
          main node/.style={}]
	
	\node[main node] (a) {$\Z^2$};
	\node[main node] (b) [below left of=a] {$\Z$};
	\node[main node] (c) [below right of=a] {$\Z$};
	\node[main node] (d) [above of=a] {$\Z$};  
	
	\path
	(d) edge node {($\id, \id$)} (a)
	(a) 
	(b) 
	edge node {($0,\id$)} (a)
	(c) 
	edge node[swap]  {($\id,0$)} (a);
	\end{tikzpicture}
      \end{figure}
      of~$\msfS_4^\star$ described in Figure~\ref{fig:D4}.

\begin{pro}\label{prop:D4}
\begin{equation*}\label{equ:D4}
 \zeta_{V(\lri)}(s) =
 \frac{1+2t^3-2t^4-t^7}{(1-t)^3(1-t^3)(1-t^5)(1-qt^4)}.
\end{equation*}
\end{pro}

\begin{proof}[Sketch of proof]
  For $a_1,a_2,a_3\in\N_{0}$, let
  $m_1=\mfp^{a_1},m_2=\mfp^{a_2},m_3=\mfp^{a_3}$. Clearly
  \begin{align*}
    \zeta_{V(\lri)}(s) &= \sum_{\substack{(a_1,a_2,a_3)\in\N_0^3,\;
        \Lambda \leq \lri^2\\ (m_1,0), (0,m_2), (m_3,m_3) \in\Lambda}}
    |\lri^2:\Lambda|^{-s}t^{a_1+a_2+a_3}\\ &=
    \frac{1}{1-t^5}\sum_{\substack{(a_1,a_2,a_3)\in\N_0^3,\; \Lambda
        \leq \lri^2 \textup{ maximal}\\ (\star) (m_1,0), (0,m_2),
        (m_3,m_3) \in\Lambda}} |\lri^2:\Lambda|^{-s}t^{a_1+a_2+a_3}.
    \end{align*}

    Assume that the maximal lattice $\Lambda\leq \lri^2$ is the row span of
    the matrix $M\in\Mat_2(\lri)$, encoding coordinates of vectors with
    respect to the standard basis. As in Section~\ref{subsec:latt} we write
    $M = D \alpha^{-1}$ for $D = \diag(\pi^{r},1)$ for $r\in\N_0$ and
    $\alpha = (\alpha_{ij}) \in\GL_2(\lri)$. Without loss of generality we may
    assume that~$r>0$. Condition ($\star$) is equivalent to
	\begin{equation*}
	\diag(m_1,m_2,m_3)\left( \begin{matrix}1&0\\0&1\\1&1\end{matrix}\right) \leq D\alpha^{-1},
	\end{equation*}
	viz.\
	\begin{equation}\label{equ:cong}
	m_1 \alpha_{11} \equiv m_2 \alpha_{21} \equiv m_3
	(\alpha_{11}+\alpha_{21}) \equiv 0 \bmod \mfp^r.
	\end{equation}
	Consider $\bfalpha=(\alpha_{11}:\alpha_{21})$ as an element of
        $\mathbb{P}^1(\lri/\mfp^r)$. Essentially we need to count solutions to
        the congruence $XY(X+Y)\equiv 0$ in $\mathbb{P}^1(\lri/\mfp^r)$. We
        proceed by a case distinction according to the reduction of $\bfalpha$
        modulo~$\mfp$.
        \begin{itemize}
        \item[$\circ$]For the $q+1-3$ points of $\mathbb{P}^1(\Fq)$ which are
          not reductions modulo $\mfp$ of solutions to this congruence, the
          congruence \eqref{equ:cong} is only satisfied if
          $m_1\equiv m_2 \equiv m_3\equiv 0\bmod \mfp^r$. Noting that there
          are $q^{r-1}$ lattices with elementary divisor type $(r,0)$ below
          every one of type $(1,0)$, this leads to a geometric progression
	$$Z_{\textup{gen}}(q,t) := \frac{1}{(1-t)^3}\cdot \frac{t^4}{1-qt^4}.$$
      \item[$\circ$] For the three solutions of $XY(X+Y) \equiv 0 \bmod \mfp$
        we obtain the following. Counting lattices ``below'' a fixed solution
        $\bfalpha_0\in\mathbb{P}^1(\Fq)$ yields the geometric progression
	$$Z_{\textup{exc}}(q,t) :=\frac{1}{(1-t)^3}\sum_{r=1}^\infty \sum_{A=1}^{r}\#
        \{x\in \lri/\mfp^{r} \mid x \equiv \bfalpha_0\bmod \mfp^r, v(x) = A
        \}t^{4r -A}.$$ But, for $1 \leq A \leq r$,
	$$ \sum_{A=1}^{r}\# \{x\in \lri/\mfp^{r} \mid x \equiv
          \bfalpha_0\bmod \mfp^r, v(x) = A \} = \begin{cases} 1, &
            \textup{ if } A=r, \\(1-q^{-1})q^{r-A}, &\textup{ if
            }A<r.\end{cases}$$ A quick calculation yields
          $Z_{\textup{exc}}(q,t) = \frac{1}{(1-t)^3} \cdot
          \frac{t^3(1-t^4)}{(1-qt^4)(1-t^3)}$.
        \end{itemize}
        We conclude by computing
	\begin{multline*}\zeta_{V(\lri)}(s) = \frac{1}{1-t^5}\left(1 + (q+1-3)
          Z_{\textup{gen}}(q,t) + 3 Z_{\textup{exc}}(q,t)\right)\\ =
          \frac{1+2t^3-2t^4-t^7}{(1-t^5)(1-t)^3(1-qt^4)(1-t^3)}.\qedhere\end{multline*}
      \end{proof}

\subsection{Kronecker quivers}\label{subsec:kronecker}

For $b\in\N_0$, consider the so-called Kronecker quiver $\msfK_b$ consisting
of two vertices and $b$ arrows between them, all in the same direction. Let
$R$ be a global ring of integers or a cDVR.

\subsubsection{$b=1$}

An $R$-representation $V$ of the Kronecker quiver $\msfK_1$
\[
\begin{tikzcd}[arrow style=tikz,>=stealth,row sep=4em]
\bullet \arrow[rr,shift left=.0ex] &&\bullet
\end{tikzcd}
\]
is given by a map $\phi\in\Hom_{R}(R^{n_1},R^{n_2})$ for $n_1,n_2\in\N_0$. The
following lemma, which is similar to~\cite[Lemma~6.1]{GSS/88}, is a simple
consequence of the rank-nullity theorem; we omit its proof.  We denote by
$\ima(\phi)^{\textup{iso}}$ the \emph{isolator} of $\ima(\phi)$ in $R^{n_2}$,
viz.\ the largest submodule $\Lambda \leq R^{n_2}$ containing $\ima(\phi)$
such that $\Lambda/\ima(\phi)$ is torsion.

\begin{lem}
  Assume that the image of $\phi$ has rank $i$. Then
  \begin{multline*}\label{equ:one.map}
    \zeta_{V}(s) = |\ima(\phi)^{\textup{iso}}:\ima(\phi)|^{-s}\cdot \\\left(
    \prod_{j=1}^i \zeta_R(2s-j+1) \right) \left( \prod_{k=i+1}^{n_2}
    \zeta_R(s-k+1)\right) \left( \prod_{l=1}^{n_1} \zeta_R(s-l+1)
    \right).
	\end{multline*}
\end{lem}

\begin{rem}\label{rem:uniform}
  We note that, if $R=\Gri$ is a global ring of integers, then the Euler
  product
  $\zeta_{V}(s) = \prod_{\mfp\in\Spec(\Gri)\setminus\{0\}}
  \zeta_{V(\Gri_{\mfp})}(s)$ is \emph{almost uniform}: there exists a single
  rational function $W_{n_1,n_2,i}(X,Y)\in\Q(X,Y)$, depending only on the rank
  vector $(n_1,n_2)$ and $i=\rk(\ima(\phi))$, such that, for almost all prime
  ideals $\mfp$ of $\Gri$, we have
$$\zeta_{V(\Gri_{\mfp})}(s) = W_{n_1,n_2,i}(q_\mfp,q_\mfp^{-s}).$$ (This is an immediate consequence of the fact that $\zeta_R(s) = \prod_{\mfp}\frac{1}{1-q_{\mfp}^{-s}}$.) We
shall see in the next section that this phenomenon is the exception, rather
than the rule, for zeta functions of representations of Kronecker quivers
$\msfK_n$ for~$n>1$.
\end{rem}

\subsubsection{$b=2$}
Consider the following $\Z$-representation $V$ of the Kronecker
quiver~$\msfK_2$
\[
\begin{tikzcd}[arrow style=tikz,>=stealth,row sep=4em]
  \Z^2 
  \arrow[rr,shift left=0.4ex, "f_1"]
  \arrow[rr,shift right=0.40ex, swap, "f_2"]
&&\Z^2
\end{tikzcd}
\]
with maps $f_1 = \id$, $f_2 = \left( \begin{matrix}0 & 1 \\ -1 &
  0 \end{matrix}\right)$.
\begin{pro}\label{pro:kron2}
  Let $\lri$ be a cDVR of odd residue field cardinality~$q$. Then
	$$\zeta_{V(\lri)}(s) = \begin{cases}
          \frac{(1+t^2)(1-t^3)}{(1-t)(1-t^2)(1-t^4)(1-qt)(1-qt^3)},&
          \textup{ if } q \equiv 1 \bmod
          (4),\\ \frac{1+t^3}{(1-t)(1-t^4)(1-qt)(1-qt^3)},
          & \textup{ if
          } q \equiv 3 \bmod (4).\end{cases}$$
\end{pro}

\begin{proof}[Sketch of proof]
  As in the proof of Proposition~\ref{prop:D4}, write
  \begin{align*}
    \zeta_{V(\lri)}(s) &= \sum_{\substack{\Lambda_1,\Lambda_2\leq\lri^2,\\
    f_i(\Lambda_1) \leq \Lambda_2, \, i\in\{1,2\}}}|\lri^2:\Lambda_1|^{-s}|\lri^2:\Lambda_2|^{-s}\\
                       &=  \frac{1}{1-t^4}\sum_{\substack{\Lambda_1,\Lambda_2\leq\lri^2,\, \Lambda_2 \textup{ maximal}\\
    (\star)   f_i(\Lambda_1) \leq \Lambda_2, \, i\in\{1,2\}}}|\lri^2:\Lambda_1|^{-s}|\lri^2:\Lambda_2|^{-s}
    \end{align*}
    Assume that the maximal lattice $\Lambda_2\leq \lri^2$ is the row span of
    the matrix $M=D \alpha^{-1}\in\Mat_2(\lri)$, with $D=\diag(\pi^r,1)$ for
    $r\in\N_0$ and $\alpha = (\alpha_{ij})\in\GL_2(\lri)$ as before. Condition
    ($\star$) is equivalent to
	$$\Lambda_1 \left( \begin{matrix} \alpha_{11}&
      -\alpha_{21}\\\alpha_{21} & \alpha_{11} \end{matrix}\right)
    \equiv 0 \bmod \mfp^r.$$ If $q \equiv 3 \bmod (4)$, then the
    matrix is invertible and the index is $q^{2r}$, leading to a
    factor $\zeta_{\lri^2}(s)\left(1 +
    (1+q^{-1})\frac{qt^3}{1-qt^3}\right) =
    \frac{1+t^3}{(1-t)(1-qt)(1-qt^3)}$.  If $q \equiv 1 \bmod (4)$,
    then since $-1$ is a square in $\mfo/\mfp$, the matrix' determinant splits into two distinct linear
    forms; a mild modification of the proof of
    Proposition~\ref{prop:D4} yields the result.
\end{proof}

\subsubsection{$b=3$}
Let $M_1 = M_2 = \Z^3$, with $\Z$-bases $(x_1,x_2,x_3)$
resp.\ $(y_1,y_2,y_3)$. Let $D$ be a non-zero integer and consider the
triple $\bff=(f_1,f_2,f_3):M_1 \rarr M_2$ defined by
$$\left( f_j(x_i) \right)_{1 \leq i,j \leq 3} =
\left( \begin{matrix} Dy_3 & y_1 & y_2 \\ y_1 & y_3 & 0 \\ y_2 & 0&
  y_1 \end{matrix} \right) =: \mcM(\bfy).$$ We thus obtain a
$\Z$-representation $V = ((M_1,M_2), \bff)$ of the Kronecker
quiver~$\msfK_3$. Note that $\det(\mcM(\bfy))$ defines the elliptic curve $E$ given by $Y^2 = X^3 - DX$.

\begin{pro}\label{pro:elliptic.curve.example}
  Let $\lri$ be a cDVR of residue field cardinality $q$ with~$(q,2D)=1$. Then
 $$\zeta_{V(\lri)}(s) = W_1(q,q^{-s}) + |E(\Fq)| W_2(q,q^{-s}),$$
 where
	\begin{align*}
	W_1(q,t) &= \frac{1 + (q+1)(t^4+t^5) + qt^9}{(1-t)(1-qt)(1-q^2t)(1-q^2t^4)(1-q^2t^5)(1-t^6)},\\
	W_2(q,t) &= \frac{(1-t^2)t^2(1+qt^5)}{(1-t)(1-qt)(1-q^2t)(1-q^2t^4)(1-q^2t^5)(1-qt^2)(1-t^6)}.
	\end{align*}
	In particular, for these rings~$\lri$, the following functional
        equation holds:
	$$\left. \zeta_{V(\lri)}(s)\right|_{q \rarr q^{-1}} =
        q^{\binom{3}{2} + \binom{3}{2}-(6+3)s} \zeta_{V(\lri)}(s).$$
\end{pro}

\begin{proof}[Sketch of proof]
	Analogous to \cite[p.~1031]{Voll/04}. For the functional
        equation, observe that
	\begin{align*}
	\left.|E(\Fq)|\right|_{q\rarr q^{-1}} &= q^{-1}|E(\Fq)|,\\
	W_1(q^{-1},t^{-1}) &= q^{6}t^9 W_1(q,t),\\
	W_2(q^{-1},t^{-1}) &= q^7 t^9 W_2(q,t).\qedhere
	\end{align*}
\end{proof}

Note that this example is not (finitely, let alone almost) uniform, as the
function $p \mapsto |E(\Fp)|$ is not.  In a very similar way, one obtains
close analogues of \cite[Theorem~3]{Voll/05}.

\section{Functional equations for a class of $\mfp$-adic
  integrals}\label{sec:new.blue} In \cite{Voll/10}, the second author
studied a family of multivariate $p$-adic integrals generalizing
Igusa's local zeta functions, proving a general ``blueprint result" on
functional equations for such integrals. In this section we prove a
generalization of this result. In Section~\ref{sec:fun.eq.quiver.rep}
we will use it to prove Theorem~\ref{thm:main.refine}.

\subsection{Informal overview}\label{subsec:inf.overview.blue} In the following we stick closely---often, to
ease comparison, verbatim or with only minimal modifications---to the
notation of~\cite[Section~2]{Voll/10}. Before we give details we
discuss, in an informal and cursory manner, the main differences.

One motivation for studying the class of integrals defined in \cite{Voll/10}
was to capture algebraically defined integer-valued ``weight functions'' on
the vertex set $\mcV_n$ of the affine Bruhat-Tits building associated with a
group of the form $\GL_n(K)$, where $K$ is a local field. For this we
considered the natural action of the group $\GL_n(\lrispec)$, where $\lrispec$
is the valuation ring of $K$, on the set $\mcV_n$. It is well-known that
$\mcV_n$ may be interpreted as the set of full lattices inside $K^n$ up to
homothety. Note that $\GL_n(\lrispec)$ is the stabilizer of the homothety
class of the ``standard lattice'' $\lrispec^n \subset K^n$ under the natural
action. By the elementary divisor theorem, orbits are parameterized by
matrices in Smith normal form. The valuations of their diagonal entries were
encoded by one set of variables (``diagonal variables'' $\bfx$), the entries
of the diagonalizing matrices by another set of variables (``matrix
variables''~$\bfy$). Crucially, the counting problems considered could all be
described by evaluating polynomials $f(\bfx,\bfy)$ which were monomial in
$\bfx$ and became ``locally monomial'' in $\bfy$ after a Hironaka resolution
of singularities. If the dependency on the matrix variables defined
subvarieties of the quotient (flag) variety $\GL_n/B$, where $B$ is a Borel
subgroup of~$\GL_n$, then the Weil conjectures for their reductions modulo the
maximal ideal of $\lrispec$ translated, after a fair bit of work, into the
desired functional equations. For details, see \cite{Voll/10}.

In the current paper we treat, more generally, $|Q_0|$-tuples of lattices
$\Lambda_\iota\subseteq K^{n_{\iota}}$, indexed by the vertices of a quiver
$\msfQ$, up to simultaneous (!)\ homothety; see
Section~\ref{sec:fun.eq.quiver.rep} for details. The $\mfp$-adic integrals
covered by our new blueprint result are specifically designed to solve
counting problems which may be expressed in terms of polynomial functions
$f((\bfx_\iota)_{\iota\in Q_0}, (\bfy_\iota)_{\iota\in Q_0})$ which are,
again, monomial in the ``diagonal variables'' $(\bfx_\iota)$ (one set of
variables for each vertex of $\msfQ$) and whose dependency on the ``matrix
variables'' $(\bfy_{\iota})$ defines projective subvarieties of the flag
variety $\bigtimes_{\iota\in Q_0}\GL_{n_\iota}/B_\iota$. In this paper's main
application of the new blueprint, viz.\ to the proof of Theorem~\ref{thm:main.refine}
in Section~\ref{sec:fun.eq.quiver.rep}, $(n_\iota)_{\iota\in Q_0}$ will be the
rank vector of the representation $V$ of the quiver~$\msfQ$.

In the case that $a := |Q_0| =1$ we all but recover the setup and results of
\cite{Voll/10}; cf.\ Remark~\ref{rem:a=1}.

\subsection{A new blueprint result}\label{subsec:new.blue}
Let $K$ be a local field with residue field characteristic~$p>0$. Let
$\lrispec = \lrispec_{K}$ denote the valuation ring of $K$, $\mfp=\mfp_{K}$ the
maximal ideal of~$\lrispec$, and $\widebar{K}$ the residue field
$\lrispec/\mfp$. The cardinality of $\widebar{K}$ will be denoted by~$q$.

For $x\in K$, let $v(x)=v_{\mfp}(x)\in\mathbb{Z}\cup\{\infty\}$ denote
the $\mfp$-adic valuation of $x$, and $|x|:=q^{-v(x)}.$ For a finite
set $\mathcal{S}$ of elements of $K$, we set
$\Vert\mathcal{S}\Vert:=\max\{|s|\mid s\in\mathcal{S}\}.$ Fix
$k,m,a\in\mathbb{N}$ and $n_1,\dots,n_a\in\N_0$. For each
$\kappa\in[k]$, let $\left(\bff_{\kappa j}\right)_{j\in J_{\kappa}}$ be a finite family of finite sets of polynomials in
$K[y_{1},\ldots,y_{m}]$, and let
$x_{1,1},\ldots,x_{1,n_{1}},\ldots,x_{a,1},\ldots,x_{a,n_{a}}$ be
independent variables. Set $n=n_{1}+\cdots+n_{a}$. Also, for each
$h\in[a]$, we fix nonnegative integers $e_{h\iota\kappa j}$ for
$\iota\in[n_{h}]$. For
$\bfI=(I_{1},\ldots,I_{a})\in\prod_{h\in[a]}\mathcal{P}([n_{h}-1])$
and $\kappa\in[k]$, we set
\[\bfg_{\kappa,\bfI}(\bfx,\bfy)=\bigcup_{j\in J_{\kappa}}\left(\prod_{h\in[a]}\prod_{\iota\in I_h^*}x_{h,\iota}^{e_{h\iota\kappa j}}\right)\bff_{\kappa
  j}(\bfy),\] where
\begin{equation}\label{equ:Ihstar}
  I_h^* = I_h \cup \{n_h\}.
  \end{equation}

Let $W\subseteq \lrispec^{m}$ be a subset which is a union of cosets modulo
$\mfp^{m} = \mfp \times \dots \times\mfp$ and
$\bfs=(s_1,\ldots,s_{k})$ be independent complex variables. With
$\lcard=\lcard(\bfI) = \sum_{h\in[a]}|I_{h}|$ and $W_a=\lrispec^{a}\setminus
\mfp^{a}$ we define

\begin{equation*}\label{eq:gr(6)}
	Z_{W,K,\bI}^{\gr}(\bfs):=\int_{\mfp^{\lcard}\times W_a\times
          W}\prod_{\kappa\in[k]}\left\Vert\bfg_{\kappa,\bfI}(\bfx,\bfy)\right\Vert^{s_{\kappa}}|\mathrm{d}\bfx_{\bI}||\mathrm{d}\bfy|,
\end{equation*}
where
$$|\mathrm{d}\bfx_{\bI}|=\left|\bigwedge_{h\in[a]} \bigwedge_{i\in
  I_h^*} dx_{h,\iota}\right|$$ is the Haar measure on $K^{\lcard+a}$
normalized so that $\lrispec^{\lcard+a}$ has measure 1 (and thus
$\mfp^{\lcard+a}$ has measure $q^{-\lcard-a}$), and
$|\mathrm{d}\bfy|=|dy_{1}\wedge\cdots\wedge dy_{m}|$ is the
(normalized) Haar measure on~$K^{m}$.

\begin{rem}\label{rem:a=1}
  Setting $a=1$ all but recovers the integral $Z_{W,K,I}(\bfs)$ defined on
  \cite[p.~1191]{Voll/10}. Indeed, in this case,
  $Z_{W,K,\bI=(I)}^{\gr}(\bfs) = (1-q^{-1})Z_{W,K,I}(\bfs)$. The factor
  $1-q^{-1}$ reflects the occurrence of the factor $W_1 = \lrispec\setminus \mfp$ in
  the domain of integration, which does not feature in the integrand. Similar
  reasoning explains the apparent mismatch between the special cases of
  Theorems~\ref{thm:grthm2.1}, \ref{thm:grthm2.2}, and~\eqref{eq:normalised}
  for $a=1$ and their respective counterparts in~\cite{Voll/10}.
\end{rem}

We now assume that the polynomials constituting the sets
$\bff_{\kappa j}(\bfy)$ are in fact defined over a number field~$F$. As in
\cite{Voll/10}, we may consider the local zeta functions $Z_{W,K,\bfI}(\bfs)$
for all non-archimedean completions $K$ of~$F$. Also, recall the definition of
a principalization $(Y,h)$ with good reduction modulo~$\mfp$. Specifically,
let $(Y,h)$, $h:Y\rightarrow\mathbb{A}^{m}$, be a principalization of the
ideal
\[\mathcal{I}=\prod_{\kappa\in[k]}\prod_{j\in J_{\kappa}}(\bff_{\kappa j}),\]  
where $(\bff)$ denotes the ideal generated by the finite set $\bff$ of
polynomials, with numerical data $(N_{t\kappa j},\nu_{t})_{t\in
  T,\,\kappa\in[k],\,j\in J_{\kappa}}$. Recall that, informally
speaking, the numerical data---both the $N_{t\kappa j}$ and the
$\nu_{t}$ are non-negative integers---keep track of the multiplicities
of the irreducible components $E_t$, $t\in T$, of the (reduced)
$h$-preimage of the scheme defined by~$\mathcal{I}$ and the
transform of the Haar measure under~$h$, respectively; for further
details, see \cite[Sec.~2.1]{Voll/10}.

\begin{thm}\label{thm:grthm2.1}
  Suppose that all the sets $\bff_{\kappa j}$ are integral (i.e.,
  contained in $\mfo[\bsy]$) and do not define the zero ideal modulo
  $\mfp_K$, and that $(Y,h)$ has good reduction modulo~$\mfp_K$. Then
	\[Z_{W,K,\bI}^{\gr}(\bfs)=\frac{(1-q^{-1})^{\lcard+a}}{q^{m}}\sum_{U\subseteq T}c_{U,W}(q)(q-1)^{|U|}\Xi_{U,\bI}(q,\bfs),\]
	where
	\[c_{U,W}(q)=\left|\{b\in \overline{Y}(\overline{K})\mid
            b\in\overline{E_{u}}(\overline{K})\Leftrightarrow u\in U
            \textup{ and
            }\overline{h}(b)\in\overline{W}\}\right|\](where
            $\overline{\phantom{x}}$ denotes reduction modulo~$\mfp$
            and
            $\overline{W}=\{(\overline{y_1},\ldots,\overline{y_m})|(y_1,\ldots,y_m)\in
            W\}$) and
	\begin{equation*}
          \Xi_{U,\bI}(q,\bfs)=\sum_{\substack{(m_{u})_{u\in U}\in\mathbb{N}^{|U|}\\(r_{h,\iota})_{h\in[a],\, \iota\in I_h}\in\mathbb{N}^{\lcard}\\(r_{h,n_{h}})\in(\mathbb{N}_{0}^{a}\setminus\mathbb{N}^{a})}}q^{-\sum_{h,\iota}^{*}r_{h,\iota}-\sum_{u}v_{u}m_{u}-\sum_{\kappa}s_{\kappa}\min_{j\in J_{\kappa}}\{\sum_{h,\iota}^{*}e_{h\iota\kappa j}r_{h,\iota}+\sum_{u}N_{u\kappa j}m_{u}\}},	
	\end{equation*}\
	where $\sum_{h,\iota}^{*}:=\sum_{h\in[a]}\sum_{\iota\in
          I_h^*}$.
\end{thm}
\begin{proof}
  The result is proved as \cite[Theorem 2.1]{Voll/10}, except that integration
  is now over $\mfp^{\lcard}\times W_a$ instead of $\mfp^{\lcard}$, yielding
  the factor $\frac{(1-q^{-1})^{\lcard+a}}{q^{m}}$ instead of
  $\frac{(1-q^{-1})^{\lcard}}{q^{m}}$.
\end{proof}
We now make the further assumption that
$m=n_{1}^{2}+\cdots+n_{a}^{2}.$ We identify
$K^{n_{1}^{2}}\times\cdots\times K^{n_{a}^{2}}$ with
$\Mat_{n_{1}}(K)\times\cdots\times\Mat_{n_{a}}(K)$ and assume that the
ideals $(\bff_{\kappa j}),\,\kappa\in[k],\,j\in J_{\kappa},$ are
$\bfB(F)$-invariant, where, for $h\in[a]$, $B_{h}(F)$ is the group of
$F$-rational points of the Borel subgroup of upper-triangular matrices
in $G_{h} = \GL_{n_{h}}$ acting on
$K[y_{1,1},y_{1,2},\ldots,y_{n_{h},n_{h}}]$ by matrix-multiplication
from the right, and $\bfB(F)=B_{1}(F)\times\cdots\times B_{a}(F).$ Let
$(\bfY, h),\,h:Y\rightarrow G_{1}/B_{1}\times\cdots\times G_{a}/B_{a}$
be a principalization of the ideal
$\mathcal{I}=\prod_{\kappa,\iota}(\bff_{\kappa j})$. Denoting, as
above, by $\mathcal{V}$ the subvariety of
$\bfG/\bfB(K)=G_{1}/B_{1}(K)\times\cdots\times G_{a}/B_{a}(K)$ defined
by $\mathcal{I}$ and by $\mathcal{V}_{\kappa j}$ the subvariety
defined by $(\bff_{\kappa j})$ yields numerical data $(N_{t\kappa
  j},v_{t})_{t\in T,\,\kappa\in[k],\,\iota\in J_{k}}$.  We study the
integral
\[Z_{\bI}^{\gr}(\bfs):=Z_{W,K,\bI}^{\gr}(\bfs)\] for
$W={\boldsymbol{\Gamma}} :=
\Gamma_{1}\times\cdots\times\Gamma_{a}=\GL_{n_{1}}(\lrispec)\times\cdots\times\GL_{n_{a}}(\lrispec)$
for almost all completions $K$ of~$F$. Note that the Haar measure $\mu'$ on
$\Gamma_{1}\times\cdots\times\Gamma_{a}$ coincides with the additive Haar
measure $\mu$ induced from
$\lrispec^{n_{1}^{2}}\times\cdots\times \lrispec^{n_{a}^{2}}$ (and normalized
such that
$\mu(\lrispec^{n_{1}^{2}}\times\cdots\times \lrispec^{n_{a}^{2}})=1$). This
implies that all the cosets of a finite-index subgroup
$\Gamma_{1}'\times\cdots\times\Gamma_{a}'\leq\Gamma_{1}\times\cdots\times\Gamma_{a}$
have measure
$\mu(\Gamma_{1})/[\Gamma_{1}:\Gamma_{1}']\cdots\mu(\Gamma_{a})/[\Gamma_{a}:\Gamma_{a}']$,
with $\mu(\Gamma_{h})=(1-q^{-1})\cdots(1-q^{-n_{h}})$ for each $h\in[a]$.
\begin{thm}\label{thm:grthm2.2} Suppose that, in addition to the above
  assumptions, none of the ideals $(\bff_{\kappa j})$ is equal to the
  zero ideal modulo~$\mfp_K$, and that $(\bfY,h)$ has good reduction
  modulo~$\mfp_K$. Then
	\[Z_{\bI}^{\gr}(\bfs)=\frac{(1-q^{-1})^{\lcard+a+n}}{q^{\sum_{h\in[a]}\binom{n_{h}}{2}}}\sum_{U\subseteq T}c_{U}(q)(q-1)^{|U|}\Xi_{U,\bI}(q,\bfs).\]
\end{thm}
\begin{proof}
  The proof is analogous to that of~\cite[Theorem 2.2]{Voll/10}. Recall that
  $\boldsymbol{\Gamma}=\Gamma_{1}\times\cdots\times\Gamma_{a}$. For each
  $h\in[a]$, we write $\Gamma_{h}$ as a disjoint union of sets
\[\Gamma_{h,\sigma_{h}}=\left\{\bfx_{h}\in\Gamma_{h} \mid \overline{\bfx}_{h}\in B_{h}(\mathbb{F}_q)\sigma_{h}B_{h}(\mathbb{F}_q)\right\},\]
$\sigma_{h}\in S_{n_{h}}$, where
$\GL_{n_{h}}(\mathbb{F}_q)=\bigcup_{\sigma_{h}\in
  S_{n_{h}}}B_{h}(\mathbb{F}_q)\sigma_{h}B_{h}(\mathbb{F}_q)$ is the
Bruhat decomposition. Write
$\boldsymbol{\sigma}=(\sigma_1,\ldots,\sigma_a)\in S_{n_{1}}\times
\dots \times S_{n_{a}}$ and
$\boldsymbol{\Gamma_{\sigma}}=\Gamma_{1,\sigma_{1}}\times\cdots\times\Gamma_{a,\sigma_{a}}.$
Thus
\begin{equation*}\label{eq:grbruhat}
  Z_{\bI}^{\gr}(\bfs)=\sum_{\boldsymbol{\sigma}\in\prod_{h\in[a]}S_{n_{h}}}Z_{\boldsymbol{\Gamma_{\sigma}},K,\bI}^{\gr}(\bfs).
\end{equation*}
There is an obvious map $\gamma:\boldsymbol{\Gamma}\rightarrow
\bfG/\bfB(K)$, and, by our invariance assumption on the ideals
$(\bff_{\kappa j})$, the value of the integrand of
$Z_{\bI}^{\gr}(\bfs)$ at a point $(\bfx,\bfy)\in \mfp^{\lcard}\times
W_a\times\boldsymbol{\Gamma}$ only depends on $\bfx$ and
$\gamma(\bfy)$. By taking the measure $\omega$ on $\bfG/\bfB(K)$ which
induces the Haar measure on the product of unit balls
$\lrispec^{\binom{n_{h}}{2}}$ of each affine chart satisfying
$\omega(a+\mfp^{\sum_{h\in[a]}\binom{n_{h}}{2}})=q^{-\sum_{h\in[a]}\binom{n_{h}}{2}}$
and noting that $\mu(B_{1}\times\cdots\times B_{a})=(1-q^{-1})^{n}$,
we obtain
\[Z_{\boldsymbol{\Gamma_{\sigma}},K,\bI}(\bfs)=(1-q^{-1})^{n}\int_{\mfp^{\lcard}\times W_a\times \boldsymbol{V_{\sigma}}}\prod_{\kappa\in[k]}\left\Vert\bfg_{\kappa,\boldsymbol{\bfI}}(\bfx,\bfy)\right\Vert^{s_{\kappa}}|\mathrm{d}\bfx_{\bI}|\mathrm{d}\omega,\]
where $\boldsymbol{V_{\sigma}}=\gamma(\boldsymbol{\Gamma_{\sigma}}).$
The projective variety $\bfG/\bfB$ may be covered by varieties
$\boldsymbol{U_{\sigma}},$ isomorphic to affine
$\sum_{h\in[a]}\binom{n_{h}}{2}$-space, indexed by the elements
$\boldsymbol{\sigma}\in\prod_{h\in[a]}S_{n_{h}}$, such that each
$\boldsymbol{V}_{\boldsymbol{\sigma}}$ is contained in
$\boldsymbol{U_{\sigma}}$ and is a union of cosets mod
$\mfp^{\sum_{h\in[a]}\binom{n_{h}}{2}}$.  The result now follows from
Theorem \ref{thm:grthm2.1}, just as \cite[Theorem 2.2]{Voll/10}
follows from \cite[Theorem~2.1]{Voll/10}. \qedhere
\end{proof}

We now consider the normalized integrals

\begin{equation*}\label{eq:normalised}
	\wt{Z_{\bI}^{\gr}}(\bfs):=\frac{Z_{\bI}^{\gr}(\bfs)}{(1-q^{-1})^{\lcard+a}\mu(\boldsymbol{\Gamma})}.
\end{equation*}

\begin{cor}\label{cor:grcor2.1}
  For $U\subseteq T$, let $E_U = \bigcap_{u\in U}E_u$ and $b_{U}(q)$
  denote the number of $\widebar{K}$-rational points of
  $\widebar{E_{U}}$. Then
  \begin{equation}\label{eq:gr11}
    \wt{Z_{\bI}^{\gr}}(\bfs)=|\bfG/\bfB(\mathbb{F}_{q})|^{-1}\sum_{U\subseteq T}b_{U}(q)\sum_{V\subseteq U}(-1)^{|U\setminus V|}(q-1)^{|V|}\Xi_{V,\bfI}(q,s).
  \end{equation}
\end{cor}

\begin{proof}
  This is analogous to \cite[Corollary~2.1]{Voll/10}. It follows immediately
  from the formula given for $Z_{\bI}^{\gr}(\bfs)$ in Theorem
  \ref{thm:grthm2.2}, Definition \ref{eq:normalised}, the fact that
  $|G_{h}/B_{h}(\mathbb{F}_q)|=\binom{n_{h}}{[n_{h}-1]}_{q}$ for $h\in[a]$ and
  from the identity
  \[c_{V}(q)=\sum_{V\subseteq U\subseteq T}(-1)^{|U\setminus V|}b_{U}(q).\]
  Thus
	\begin{align*}
		\wt{Z_{\bI}^{\gr}}(\bfs)&:=\frac{Z_{\bI}^{\gr}(\bfs)}{(1-q^{-1})^{\lcard+a}\mu(\boldsymbol{\Gamma})}\\
		&=\frac{(1-q^{-1})^{n}}{q^{\sum_{h\in[a]}\binom{n_{h}}{2}}\mu(\boldsymbol{\Gamma})}\sum_{U\subseteq T}c_{U}(q)(q-1)^{|U|}\Xi_{U,\bI}(q,\bfs)\\
		&=\left(\prod_{h\in[a]}\left(\frac{(1-q^{-1})^{n_{h}}}{q^{\binom{n_{h}}{2}}\mu(\Gamma_{h})}\right)\right)\sum_{U\subseteq T}c_{U}(q)(q-1)^{|U|}\Xi_{U,\bI}(q,\bfs)\\
		&=\left(\prod_{h\in[a]}|G_{h}/B_{h}(\mathbb{F}_{q})|^{-1}\right)\sum_{U\subseteq T}c_{U}(q)(q-1)^{|U|}\Xi_{U,\bI}(q,\bfs)\\
		&=|\bfG/\bfB(\mathbb{F}_{q})|^{-1}\sum_{U\subseteq T}b_{U}(q)\sum_{V\subseteq U}(-1)^{|U\setminus V|}(q-1)^{|V|}\Xi_{V,\bI}(q,s).\qedhere
	\end{align*}
	\end{proof}

\begin{pro}\label{pro:grpro2.1}
  Let $L_{\sigma\tau}(\bfr),\sigma\in[s],\tau\in[t]$, be
  $\mathbb{Z}$-linear forms in independent variables
  $r_{1},\ldots,r_{\lcard}$, $r_{\lcard+1},\ldots,r_{\lcard+a}$ and
  $X_{1},\ldots,X_{\lcard+a}$,$Y_{1},\ldots,Y_{s}$ independent
  variables.  For $\bfr\in\N_0^{\lcard+a}$ set $$\bfX^{\bfr} =
  \prod_{\rho\in[\lcard]}X_{\rho}^{r_{\rho}}\prod_{\iota\in]\lcard,\lcard+a]}X_{\iota}^{r_{\iota}}
      \quad \textup{ and }\quad
      m_{\bfr}(\bfY)=\prod_{\sigma\in[s]}Y_{\sigma}^{\min_{\tau\in[t]}\{L_{\sigma\tau}(\bfr)\}}.$$
 Define further
  \begin{equation*}
    Z^{\circ}(\bfX,\bfY)=
    \sum_{\bfr\in\mathbb{N}^{\lcard}\times(\mathbb{N}_{0}^{a}\setminus\mathbb{N}^{a})}\bfX^{\bfr}m_\bfr(\bfY)\quad\text{and}\quad
    Z(\bfX,\bfY) =
    \sum_{\bfr\in\mathbb{N}_{0}^{\lcard}\times(\mathbb{N}_{0}^{a}\setminus\mathbb{N}^{a})}\bfX^{\bfr}m_\bfr(\bfY).
	\end{equation*}
Then
\[Z^{\circ}(\bfX^{-1},\bfY^{-1})=(-1)^{\lcard+a-1}Z(\bfX,\bfY).\]
\end{pro}

\begin{proof}
  Let \begin{equation*} Z^{\circ}_{1}(\bfX,\bfY)
    =\sum_{\bfr\in\mathbb{N}^{\lcard}\times\mathbb{N}_{0}^{a}}\bfX^{\bfr}
    m_\bfr(\bfY)\quad \textup{and} \quad
    Z^{\circ}_{2}(\bfX,\bfY)=\sum_{\bfr\in\mathbb{N}^{\lcard}\times\mathbb{N}^{a}}\bfX^{\bfr}
    m_\bfr(\bfY),
\end{equation*}
so that
\begin{align*}
Z^{\circ}(\bfX,\bfY)=Z^{\circ}_{1}(\bfX,\bfY)-Z^{\circ}_{2}(\bfX,\bfY).
	\end{align*}
	By \cite[Proposition 2.1]{Voll/10}
	\[Z^{\circ}_{2}(\bfX^{-1},\bfY^{-1})=(-1)^{\lcard+a}Z_{2}(\bfX,\bfY),\]
where
\[Z_{2}(\bfX,\bfY) := \sum_{\bfr\in\mathbb{N}_{0}^{\lcard}\times\mathbb{N}_{0}^{a}}\bfX^{\bfr}
                 m_\bfr(\bfY).\]
	
Similarly, \cite[Proposition 2.1]{Voll/10} and the inclusion-exclusion
principle give
	
\begin{align*}
 Z^{\circ}_{1}(\bfX^{-1},\bfY^{-1})&=\sum_{\bfr\in\mathbb{N}^{\lcard}\times\mathbb{N}_{0}^{a}}
 \bfX^{-\bfr}
 m_\bfr(\bfY^{-1})\\ &=\sum_{J\subseteq[a]}\sum_{\substack{\boldsymbol{r}\in\mathbb{N}^{\lcard}\times
     \N_0^{a}\\r_{l+j}=0 \textup{ iff.\ }j\not\in J}} \bfX^{-\bfr}
 m_\bfr(\bfY^{-1}) \\ &=\sum_{J\subseteq[a]}(-1)^{\lcard+|J|}
 \sum_{\substack{\boldsymbol{r}\in\mathbb{N}_0^{\lcard+a}\\r_{l+j}=0
     \textup{ if }j\not\in J}} \bfX^{\bfr}
 m_\bfr(\bfY)\\ &=(-1)^{\lcard+a}\sum_{\bfr\in\mathbb{N}_{0}^{\lcard}\times\mathbb{N}^{a}}\bfX^{\bfr}m_\bfr(\bfY).
	\end{align*}
	
Thus, indeed,
\begin{align*}
  Z^{\circ}(\bfX^{-1},\bfY^{-1})&=Z^{\circ}_{1}(\bfX^{-1},\bfY^{-1})-Z^{\circ}_{2}(\bfX^{-1},\bfY^{-1})\\ &=(-1)^{\lcard+a}\left(\sum_{\bfr\in\mathbb{N}_{0}^{\lcard}\times\mathbb{N}^{a}}\bfX^{\bfr}m_\bfr(\bfY)\right)-(-1)^{\lcard+a}\left(\sum_{\bfr\in\mathbb{N}_{0}^{\lcard}\times\mathbb{N}_{0}^{a}}\bfX^{\bfr}m_\bfr(\bfY)\right)
  \\ &=(-1)^{\lcard+a-1}\sum_{\bfr\in\mathbb{N}_{0}^{\lcard}\times(\mathbb{N}_{0}^{a}\setminus\mathbb{N}^{a})}\bfX^{\bfr}m_\bfr(\bfY)
  = (-1)^{\lcard+a-1}Z(\bfX,\bfY).\qedhere
	\end{align*}
\end{proof}

\begin{cor}\label{cor:grcor2.2} For all
  $\bfI\in\prod_{h\in[a]}\mathcal{P}([n_{h}-1])$, $V\subseteq T$,
	\begin{equation*}
		\left.\Xi_{V,\bI}(q,\bfs)\right|_{q\rightarrow q^{-1}}=(-1)^{|V|+\lcard+a-1}\sum_{W\subseteq V,\, \bJ\subseteq\bI}\Xi_{W,\bJ}(q,\bfs).
	\end{equation*}
\end{cor}
We record the following simple fact, whose proof is a simple computation. 
\begin{lem}
  For all $U\subseteq T,\bJ\in\prod_{h\in[a]}\mathcal{P}([n_{h}-1]),$
	\begin{equation*}\label{eq:gr15}
		\sum_{V\subseteq U}(-1)^{|U\setminus V|}(1-q^{-1})^{|V|}\sum_{W\subseteq V}\Xi_{W,\bJ}(q,\bfs)=q^{-|U|}\sum_{V\subseteq U}(-1)^{|U\setminus V|}(q-1)^{|V|}\Xi_{V,\bJ}(q,\bfs).
	\end{equation*}
\end{lem}
The following definition is analogous to \cite[(13)]{Voll/10}:
\begin{equation}\label{eq:gr13}
	b_{U}(q^{-1}):=q^{-\left(\sum_{h\in[a]}\binom{n_{h}}{2}-|U|\right)}b_{U}(q).
\end{equation}

The following set of ``inversion properties'' generalizes
\cite[Theorem~2.3]{Voll/10}.
\begin{thm}[inversion properties]\label{thm:grIPThm2.3}
  Under the assumption of Theorem \ref{thm:grthm2.2}, for all
  $\bI \in\prod_{h\in[a]}\mcP([n_{h}-1])$,
	\begin{equation*}\label{eq:grIP}
		\left.\wt{Z_{\bI}^{\gr}}(\bfs)\right|_{q\rightarrow
                  q^{-1}}=(-1)^{\lcard+a-1}\sum_{\bJ\subseteq\bI}\wt{Z_{\bJ}^{\gr}}(\bfs).
	\end{equation*}
\end{thm}
\begin{proof}
  Starting with the expression \eqref{eq:gr11} for
  $\wt{Z_{\bfI}^{\gr}}(\bfs)$ in Corollary~\ref{cor:grcor2.1} we
  obtain, by combining Corollary \ref{cor:grcor2.2}, \eqref{eq:gr13},
  and \eqref{eq:gr15}, that indeed
\begin{align*}
 \left.\wt{Z_{\bI}^{\gr}}(\bfs)\right|_{q\rightarrow
   q^{-1}}&=\left(\frac{q^{\sum_{h\in[a]}
     \binom{n_h}{2}}}{|\bfG/\bfB(\Fq)|}\right)\sum_{U\subseteq
   T}b_{U}(q^{-1})\sum_{V\subseteq U}(-1)^{|U\setminus
   V|}(1-q^{-1})^{|V|}\\ &\quad \cdot(-1)^{\lcard+a-1}\sum_{W\subseteq
   V,\bJ\subseteq\bI}\Xi_{W,\bJ}(q,\bfs)\\ &=(-1)^{\lcard+a-1}\sum_{\bJ\subseteq\bI}|\bfG/\bfB(\Fq)|^{-1}\sum_{U\subseteq
   T}q^{|U|}b_{U}(q)\sum_{V\subseteq U}(-1)^{|U\setminus
   V|}(1-q^{-1})^{|V|}\\ &\quad \cdot\sum_{W\subseteq
   V}\Xi_{W,\bJ}(q,\bfs)\\ &=(-1)^{\lcard+a-1}\sum_{\bJ\subseteq\bI}|\bfG/\bfB(\mathbb{F}_{q})|^{-1}\sum_{U\subseteq
   T}b_{U}(q)\sum_{V\subseteq U}(-1)^{|U\setminus
   V|}(q-1)^{|V|}\Xi_{V,\bJ}(q,\bfs)\\ &=(-1)^{\lcard+a-1}\sum_{\bJ\subseteq\bI}\wt{Z_{\bJ}^{\gr}}(\bfs). \qedhere \end{align*}
\end{proof}

\subsection{Local functional equations}\label{subsec:local.fun.eq}
For $\bI=(I_1,\dots,I_a)\in\prod_{h\in[a]}\mcP([n_{h}-1])$ we define
$\bI^{\textup{c}}=(I_{1}^{\textup{c}},\ldots,I_{a}^{\textup{c}})\in\prod_{h\in[a]}\mcP([n_{h}-1])$. By
the same slight abuse of notation, we extend other set-theoretic operations
componentwise when we write, for instance, $\bfI \cup \bfJ$ for
$(I_1\cup J_1,\dots,I_a\cup J_a)$, where
$\bfI,\bfJ\in\prod_{h\in[a]}\mcP([n_{h}-1])$. Assume throughout this section
that the assumptions of Theorem~\ref{thm:grthm2.2} are satisfied. The
Inversion Properties of Theorem~\ref{thm:grIPThm2.3} imply the following
result, which is analogous to \cite[Lemma~7]{VollBLMS/06}:
\begin{lem}\label{lem:L7}For all $\bI\in\prod_{h\in[a]}\mcP([n_{h}-1])$
	\[\left.\sum_{\bJ\supseteq\bI}\wt{Z_{\bJ}^{\gr}}(\bfs)\right|_{q\rightarrow q^{-1}}=(-1)^{n-1}\sum_{\bJ\supseteq\bI^{\textup{c}}}\wt{Z_{\bJ}^{\gr}}(\bfs).\]
\end{lem}
\begin{proof} Recall the identity $n=\sum_{h\in[a]}n_{h}$. Set $0^0=1$.  By Theorem
  \ref{thm:grIPThm2.3}
	\begin{align*}
		\left.\sum_{\bJ\supseteq\bI}\wt{Z_{\bJ}^{\gr}}(\bfs)\right|_{q\rightarrow q^{-1}}&=\sum_{\bJ\supseteq\bI}(-1)^{(\sum_{h\in[a]}|J_{h}|)+a-1}\sum_{\bfS\subseteq\bJ}\wt{Z_{\bfS}^{\gr}}(\bfs) = \sum_{\bfR\in\prod_{h\in[a]}\mcP([n_{h}-1])}C_{\bfR}\wt{Z_{\bfR}^{\gr}}(\bfs),
	\end{align*}
	where, for each $\bfR$ (and setting $0^0=1$),
	\begin{align*}
          C_{\bfR}&=\sum_{\bJ\supseteq(\bfR\cup\bI)}(-1)^{(\sum_{h\in[a]}|J_{h}|)+a-1}\\
                  &=(-1)^{(\sum_{h\in[a]}|R_{h}\cup I_{h}|)+a-1}\sum_{\bfS\subseteq (\bfR \cup \bfI)^{\boldsymbol{c}}}(-1)^{\sum_{h\in[a]}|S_{h}|}\\
                  &=(-1)^{(\sum_{h\in[a]}|R_{h}\cup I_{h}|)+a-1}0^{\sum_{h\in[a]}|(R_{h}\cup I_{h})^{\textup{c}}|}\\
                  &= (-1)^{n-1} \delta_{\bfR\supseteq\bI^{\textup{c}}}.
	\end{align*}
	Indeed,
	\begin{align*}
		\bR\supseteq\bI^{\textup{c}}&\iff\forall h\in[a]:R_{h}\supseteq I_{h}^{\textup{c}}\\
		&\iff\forall h\in[a]:R_{h}\cup I_{h}=[n_{h}-1]\\
		&\iff\forall h\in[a]:(R_{h}\cup I_{h})^{\textup{c}}=\emptyset\\
		&\iff\sum_{h\in[a]}|(R_{h}\cup I_{h})^{\textup{c}}|=0\\ & \iff \sum_{h\in[a]} (R_h\cup I_h)| = n-a.\qedhere
	\end{align*}
\end{proof}
Now we prove the functional equation. Given an element $w\in S_{b}$ of
the symmetric group on letters ${1,\ldots,b}$ we write
\[\textrm{Des}(w):=\{i\in[b-1]|w(i)>w(i+1)\}\]
for the \emph{descent type} of $w$. Similarly, given
$\boldsymbol{w}=(w_1,\ldots,w_a)\in\prod_{h\in[a]}S_{n_{h}}$ we write
\[\textrm{Des}(\boldsymbol{w}):=(\textrm{Des}(w_{1}),\ldots,\textrm{Des}(w_{a}))\in\prod_{h\in[a]}\mcP([n_{h}-1])\]
for the \emph{descent type} of $\boldsymbol{w}$. By $\ell(w_{h})$ we
denote the Coxeter length of $w_{h}\in S_{n_{h}}$, i.e.\ the length of
a shortest word for $w_{h}$ in the standard generators for $S_{n_{h}}$, by
$w_{0,h}$ the longest element in $S_{n_{h}}$, both for $h\in[a]$. We also
set
$\boldsymbol{w_{0}}=(w_{0,1},\ldots,w_{0,a})\in\prod_{h\in[a]}S_{n_{h}}.$
We recall the standard identities
(cf.\ \cite[Section~1.8]{Humphreys/90})
\begin{equation*}
	\textrm{Des}(w_{h}w_{0,h})=\textrm{Des}(w_h)^{\textup{c}},\quad
	\ell(w_{h})+\ell(w_{h}w_{0,h})=\binom{n_{h}}{2}=\ell(w_{0,h}).
\end{equation*}
By slight abuse of notation we write
$\ell(\boldsymbol{w})=\sum_{h\in[a]}\ell(w_{h})$, specifically
$$\ell(\boldsymbol{w_{0}})=\sum_{h\in[a]}\ell(w_{0,h}) = \sum_{h\in[a]}
\binom{n_h}{2}$$ and
$\ell(\boldsymbol{ww_{0}})=\sum_{h\in[a]}\ell(w_{h}w_{0,h})$.  Let
\begin{equation}\label{def:Z.tilde}
  \wt{Z^{\gr}}(\bfs)=\sum_{\bI \in \prod_{h\in[a]}\mathcal{P}([n_{h}-1])}\binom{\underline{n}}{\bfI}_{q^{-1}}\wt{Z_{\bI}^{\gr}}(\bfs),
  \end{equation}
where, for $\bfI = (I_1,\dots,I_h)$, we define
$$\binom{\underline{n}}{\bfI}_X :=
\prod_{h\in[a]}\binom{n_{h}}{I_{h}}_{X}$$ in terms of $X$-multinomial
coefficients: if $I_{h}=\{i_{h,1},\ldots,i_{h,l_{h}}\}_{<}\subseteq[n_{h}-1]$,
then
\begin{equation}\label{eq:multi.des}
	\binom{n_{h}}{I_{h}}_{X}=\binom{n_{h}}{i_{h,l_{h}}}_{X}\binom{i_{h,l_{h}}}{i_{h,l_{h}-1}}_{X}\cdots\binom{i_{h,2}}{i_{h,1}}_{X}=\sum_{\substack{w\in S_{n_{h}}\\\textrm{Des}(w)\subseteq I_{h}}}X^{\ell(w)}\in\Z[X].
\end{equation}

\begin{thm}[local functional equations]\label{thm:grfuneq}
 \begin{equation*}\label{eq:gr(2.13)} \left.\wt{Z^{\gr}}(\bfs)\right|_{q\rightarrow q^{-1}}=(-1)^{n-1}q^{\sum_{h\in[a]}\binom{n_{h}}{2}}\wt{Z^{\gr}}(\bfs).
 \end{equation*}
\end{thm}

\begin{proof}
  Using Lemma~\ref{lem:L7} and the Coxeter-group theoretic facts collected
  above we obtain
	\begin{align*}
          \left.\wt{Z^{\gr}}(\bfs)\right|_{q\rightarrow
          q^{-1}}&=\sum_{\bI \in \prod_{h\in[a]}\mathcal{P}([n_{h}-1])}\left.\binom{\underline{n}}{\bfI}_{q^{-1}}\wt{Z_{\bI}^{\gr}}(\bfs)\right|_{q\rightarrow
                   q^{-1}}&\text{by \eqref{def:Z.tilde}}\\ &=\sum_{\boldsymbol{w}\in\prod_{h}S_{n_{h}}}q^{-\ell(\boldsymbol{w})}\left.\sum_{\bJ\supseteq\textrm{Des}(\boldsymbol{w})}\wt{Z_{\bJ}^{\gr}}(\bfs)\right|_{q\rightarrow
                               q^{-1}}&\text{by \eqref{eq:multi.des}}\\ &=q^{\sum_{h\in[a]}\binom{n_{h}}{2}}\sum_{\boldsymbol{w}\in\prod_{h}S_{n_{h}}}q^{-\ell(\boldsymbol{ww_{0}})}(-1)^{n-1}\sum_{\bJ\supseteq\textrm{Des}(\boldsymbol{ww_{0}})}\wt{Z_{\boldsymbol{J}}^{\gr}}(\bfs)&\\ &=(-1)^{n-1}q^{\sum_{h\in[a]}\binom{n_{h}}{2}}\wt{Z^{\gr}}(\bfs).&\qedhere
	\end{align*}
\end{proof}

\begin{rem}
  One may compare Theorem~\ref{thm:grfuneq} with \cite[Theorem~3.8]{CSV/19}
  (or, equivalently, \cite[Theorem~1.6]{CSV_FPSAC2020}). This result
  establishes that ``generalized Igusa functions''---combinatorially defined
  rational functions introduced in
  \cite[Definition~3.5]{CSV/19}(=\cite[Definition~1.5]{CSV_FPSAC2020})---satisfy
  functional equations akin to those established in
  Theorem~\ref{thm:grfuneq}. Its proof rests on the technical
  \cite[Proposition~3.10]{CSV/19}, an apparent analogue of
  Lemma~\ref{lem:L7}. It remains a challenge to determine whether generalized
  Igusa functions may be expressed via (monomial) $p$-adic integrals that fit
  into the remit of the $\mfp$-adic methodology developed in this section.
\end{rem}

\section{Functional equations for local quiver representation zeta
  functions}\label{sec:fun.eq.quiver.rep}
In this section we apply Theorem \ref{thm:grfuneq} to prove Theorem~\ref{thm:main.refine}, a functional
equation for the generic multivariate  local zeta functions associated with nilpotent
integral quiver representations satisfying the homogeneity
condition~\ref{cond:quiver.hom}.

\subsection{Informal overview}\label{subsec:inf.overview.funeq}
The arguments developed in this section are largely analogous to those of
\cite{VollIMRN/19}. To facilitate comparison, we follow the notation
and terminology from \cite{VollIMRN/19} closely. We briefly discuss the main differences.

Throughout, let $\Gri$ be a global ring of integers, $\mfp$ be a non-zero
prime ideal of $\Gri$. We write $\lri = \Gri_\mfp$ and $K=K_{\mfp}$ for the
field of fractions of~$\lri$.

As discussed in Section~\ref{subsubsec:submod}, the setup of \cite{VollIMRN/19}
are nilpotent representations of loop quivers. In particular, all arrows have
the same head and tail, namely the unique vertex, represented by a single
module $\mcL(\lri)\cong \lri^n$. The task of enumerating submodules is
approached by controlling the simplicial subcomplex they define of the
Bruhat-Tits building associated with the $\mfp$-adic group
$\GL_n(K_\mfp)$. The vertices of this complex, viz.\ homothety classes of full
$\mfp$-adic lattices inside~$K_\mfp^n$, are parameterized by means of the
action of the group $\Gamma=\GL_n(\lri)$ on the building. The technical
challenge overcome in \cite{VollIMRN/19} was to describe the submodule
condition in terms of polynomial functions on $\Gamma$ which, if
``homogeneity'' holds, were amenable to the $\mfp$-adic integration machinery
of~\cite{Voll/10}.

In the current paper we consider subrepresentations of nilpotent
representations of general quivers. Here, the need to keep track of heads and
targets of various arrows ramps up complexity. Indeed, instead of a single
lattice, we consider compatible tuples $(\Lambda_\iota)_{\iota\in Q_0}$ of
lattices $\Lambda_\iota\leq \mcL_\iota(\lri)\cong \lri^{n_\iota}$. In analogy
to the approach in \cite{VollIMRN/19} for $|Q_0|=1$, we express the
subrepresentation condition in terms of polynomial functions on
$\prod_{\iota\in Q_0}\Gamma_\iota$, where
$\Gamma_\iota = \GL_{n_\iota}(\lri)$. The generalization of \cite{Voll/10}
developed in Section~\ref{sec:new.blue} is tailor-made to deal with these
functions provided the homogeneity condition~\ref{cond:quiver.hom} holds.

To prove Theorem~\ref{thm:main.refine} we are looking to establish the
functional equation~\eqref{equ:funeq.refine} for almost all zeta
functions $\zeta_{V(\lri)}(\bfs)$, enumerating the finite-index
$\lri$-sub\-re\-pre\-sen\-ta\-tions $V'$ of~$V(\lri)$, written $V'\leq
V(\lri)$. Recall that each such subrepresentation is of the form
$V'=(\Lambda_{\iota},f'_\phi)_{\iota\in Q_0, \phi\in Q_1}$, for
$\lri$-modules $\Lambda_{\iota}$ of ranks $n_{\iota}$. We write
$\bfLambda = (\Lambda_{\iota})_{\iota\in Q_0}$. Note that such tuples
may be identified with \emph{graded} $\lri$-submodules of
$\bmcL(\lri)=\bigoplus_{\iota\in Q_{0}}\mcL_{\iota}(\lri)$
(see~\eqref{equ:bmcL}) via the map $(\Lambda_{\iota})_{\iota\in Q_0}
\mapsto \bigoplus_{\iota\in Q_0} \Lambda_{\iota}$.  Clearly the
property of being the support of a subrepresentation is really a
property of the integral members of the (simultaneous!)
\textit{homothety class} $[\bsLambda]=\{x\bsLambda\mid x\in
K_\mfp^{*}\}$ of $\bsLambda$ in $(K_\mfp^{n_{\iota}})_{\iota\in Q_0}$:
either all elements of $[\bsLambda]$ support $V(\lri)$-representations
or none does. By slight abuse of notation we write $[\bsLambda]\leq
V(\lri)$ in the former case and set
\[\SubRep_{V(\lri)}=\{[\bsLambda]\mid \bfLambda \leq V(\lri)\}.\]
 Evidently, every homothety class $[\bsLambda]$ of $a$-tuples of
 $\lri$-sublattices $\Lambda_{\iota}\subseteq K_\mfp^{n_{\iota}}$
 contains a unique maximal integral element $\bsLambda_{\max}$,
 i.e.\ $\bfLambda_{\max}\leq \mcL(\lri)$, but
 $\mfp^{-1}\bfLambda_{\max}\not\leq \mcL(\lri)$. As the intersection
 of $[\bsLambda]$ with the set of all $a$-tuples of $\lri$-sublattices
 equals $\{\mfp^{m}\bsLambda_{\max}\mid m\in\mathbb{N}_{0}\}$ it thus
 suffices---in principle---to describe the elements of
 $\SubRep_{V(\lri)}$ and to control their maximal integral members'
 indices in $\bmcL(\lri)$. Recall that $a = |Q_0|$. Indeed,
\begin{equation}
	\zeta_{V(\lri)}(\bfs)=\frac{1}{1-q^{-\sum_{h\in [a]}n_{h}s_{h}}}\sum_{\substack{[\bsLambda]\in\SubRep_{V(\lri)}\\\bsLambda=\bsLambda_{\max}}}\prod_{h\in[a]}|\mcL_{h}(\lri):\Lambda_{h}|^{-s_{h}}.\label{equ:zeta.refine}
\end{equation}
For each $h\in[a]$, keeping track of the indices
$|\mcL_{\iota}(\lri):\Lambda_{\iota}|$ for each unique maximal element
$\bsLambda=\bsLambda_{\max}$ is easy (see~\eqref{eq:index.refine}), so
the problem of computing the right-hand side
of~\eqref{equ:zeta.refine} is to identify $\SubRep_{V(\lri)}$ as a
subset of the set $\bmcV$ of all homothety classes~$[\bfLambda]$.

In the case $a=1$---the case treated in \cite{VollIMRN/19}---this set of
homothety classes may be identified with the vertices of the affine
Bruhat-Tits building associated with the group~$\GL_n(K_\mfp)$. In general,
the disjoint union of the buildings associated with the groups
$\GL_{n_{\iota}}(K_\mfp)$ may serve as a geometric model. We will not pursue
this vantage point.

The proof now follows the lines of that of
\cite[Theorem~1.2]{VollIMRN/19}, with $\SubRep$, $\bfLambda$, and
$\bmcV$ taking the places of $\SubMod$, $\Lambda$, and $\mathcal{V}_n$, respectively.

\subsection{Cocentral bases}
We identify $Q_0$ with $[a]$. For $h\in[a]$ and $i\in[c]_{0}$ we
write, as introduced in
Section~\ref{subsec:nil.qui}, $$n_{h,i}=\rk_\Gri\mcL_{h,i},\quad
N_{h,i}=\sum_{j\leq c-i}n_{h,j}, \quad \textup{and} \quad
N_{i}=\sum_{h\in[a]}N_{h,i}.$$ An $\Gri$-basis
$\bse_{h}=(e_{h,1},\ldots,e_{h,n_{h}})$ of $\mcL_{h}$ is called
\textit{cocentral} if
\[Z_{h,i} = Z_i \cap \mcL_h = \langle e_{h,N_{h,i}+1},\ldots,e_{h,n_{h}}\rangle_{\Gri}\] 
for all $i\in[c]$. An $\Gri$-basis $\bse=((\bse_1),\ldots,(\bse_a))$
of $\bmcL$ is called \textit{cocentral} if $\bse_{h}$ is cocentral for
all $h\in[a]$. By Assumption \ref{ass:quiver.cocentral}, cocentral
bases clearly exist. Condition \ref{cond:quiver.hom} is equivalent to
the following condition.

\begin{cond}\label{cond:quiver.hom.matrix}
  There exist generators $c_{1},\ldots,c_{d}$ of $\mcE$ and a cocentral
  $\Gri$-basis $\bse$ of $\bmcL$ such that, for all $k\in[d]$, the matrix
  $C_{k}$ representing $c_{k}$ with respect to $\bse$ (acting from the right
  on row vectors) has the form
  \[C_{k}=\left(C_{k}^{(th)}\right)_{t,h\in[a]}\in\Mat_{n}(\Gri)\] for blocks
  $C_{k}^{(th)}$ which have the form
  \begin{equation}\label{def:Ckth}
    C_{k}^{(th)}=\left(\left(C_{k}^{(th)}\right)^{(ij)}\right)_{i,j\in[c]}\in\Mat_{n_{t}\times
      n_{h}}(\Gri)
  \end{equation} for blocks
  $\left(C_{k}^{(th)}\right)^{(ij)}\in\Mat_{n_{t,i}\times
    n_{h,j}}(\Gri)$ which are zero unless $j=i+1$.
\end{cond}	
\begin{rem}\label{rem:cond}
Condition \ref{cond:quiver.hom} is equal to \cite[Condition
  1.1]{VollIMRN/19}, but Condition~\ref{cond:quiver.hom.matrix} is a
proper generalization of \cite[Condition 2.1]{VollIMRN/19}, to which it
specializes in the case $a=|Q_0|=1$.
\end{rem}

\subsection{Lattices, matrices, and the subrepresentation condition}\label{subsec:latt}
Let $\bse$ be a cocentral $\Gri$-basis of $\bmcL$ as in Condition
\ref{cond:quiver.hom.matrix}. It yields an $\lri$-basis of
$\bmcL(\lri)$ which we also denote by $\bse$ and which allows us to
identify $\bmcL(\lri)$ with $\lri^{n}$ and $\mcE(\lri)$ with matrices
$C_{1},\ldots,C_{d}$ representing the $\lri$-linear operators
$c_{1}\ldots,c_{d}$.

As in Section~\ref{subsec:new.blue} we write $\Gamma_h=\GL_{n_{h}}(\lri)$ for
$h\in[a]$ and set $\bsGamma = \Gamma_1 \times \dots \times \Gamma_a$. A full
$\lri$-sublattice $\boldsymbol{\Lambda}$ of $\bmcL(\lri)$ may be identified
with a coset $\bsGamma\bsM$ for a matrix
$\bsM\in\mathrm{GL}_{n}(K_\mfp)\cap\mathrm{Mat}_{n}(\lri)$, whose rows encode
the coordinates with respect to $\bse$ of a set of generators of
$\boldsymbol{\Lambda}$. If $\boldsymbol{\Lambda}$ is graded, then $\bsM$ is a
block diagonal matrix
\begin{equation}\label{def:M}\bsM=\diag\left( M_1,\dots,M_a\right),
  \end{equation}
where $M_h\in\GL_{n_{h}}(K_\mfp)\cap \Mat_{n_{h}}(\lri)$ for all
$h\in[a]$. If $\bfLambda = (\Lambda_h)_{h\in[a]}$, then the coset
$\Gamma_h M_{h}$ is identified with the full $\lri$-sublattice
$\Lambda_{h}$ of $\mcL_{h}(\lri)$.  Let $\pi$ be a uniformizer of
$\lri$. By the elementary divisor theorem, for each $h\in[a]$ there
exist
$$I_{h}=\{i_{h,1},\ldots,i_{h,l_{h}}\}_{<}\subseteq[n_{h}-1],\quad
r_{h,n_{h}}\in\mathbb{N}_{0}, \quad
\bold{r}_{h,I_{h}}=(r_{h,i_{1}},\ldots,r_{h,i_{l_{h}}})\in\mathbb{N}^{l_{h}},$$
all uniquely determined by $\Lambda_{h}$, and
$\alpha_{h}\in\Gamma_{h}$ such that $M_{h}=D_{h}\alpha_{h}^{-1}$,
where

\begin{multline}\label{def:D}
	D_{h}=\pi^{r_{h,n_{h}}}\diag\left((\pi^{\sum_{\iota\in
			I_{h}}r_{h,\iota}})^{(i_{h,1})},(\pi^{\sum_{\iota\in
			I_{h}\setminus\{i_{h,1}\}}r_{h,\iota}})^{(i_{h,2}-i_{h,1})},\ldots,\right.\\
		\left.(\pi^{r_{h,i_{l_{h}}}})^{(i_{h,l_{h}}-i_{h,l_{h-1}})},1^{(n_{h}-i_{h,l_{h}})}\right)\in\Mat_{n_{h}}(\lri).
	\end{multline}

        We write $\nu([\Lambda_{h}])=(I_{h},\bold{r}_{h,I_{h}})$.  Note that
        $r_{h,n_{h}}=v(M_h)$, the $\mfp$-adic valuation of the matrix
        $M_{h}$. We also write $$\nu(\bfLambda)=(\bfI,\bfr),$$ where
$$\bfI=(I_{1},\ldots,I_{a})\in\prod_{h\in[a]}\mcP([n_{h}-1])$$ and,
setting $\lcard = \sum_{h\in[a]} l_h = \sum_{h\in[a]}|I_h|$ as in
Section~\ref{subsec:new.blue},
$$\bfr=(\bfr_{1,I_{1}},\ldots,\bfr_{a,I_{a}},r_{1,n_{1}},\ldots,r_{a,n_{a}})\in\mathbb{N}^{l_{1}}\times\cdots\times\mathbb{N}^{l_{a}}\times\mathbb{N}_{0}^{a}
= \N^{\lcard}\times \N_0^a.$$ Recall that $I_{h}^{*}:=I_{h}\cup\{n_{h}\}$ for
$h\in[a]$. Obviously, for each $h\in[a]$
\begin{equation}\label{eq:index.refine}
	|\mcL_{h}(\lri):\Lambda_{h}| = q^{v(\det D_h)} =
	q^{\sum_{\iota\in
			I_h^*}\iota r_{h,\iota}}.
\end{equation}

We call $\bfLambda$ \emph{maximal} if
$\bfr\in\mathbb{N}^{\lcard}\times
(\mathbb{N}_{0}^{a}\setminus\mathbb{N}^{a})$ and denote by
$\bsLambda_{\max}$ the unique maximal element of $[\bsLambda]$. We set
$$\nu([\bsLambda]) := \nu(\bsLambda_{\max})\in \prod_{h=1}^a \mcP([n_h-1])
\times \left(\N^{\lcard}\times (\N_0^a\setminus\N^a)\right).$$

In the sequel we will often toggle between lattices $\bsLambda$ (resp.\
$\Lambda_{h}$) and representing matrices $\bsM$ (resp.\ $M_{h}$), extending
notation for lattices to matrices representing them. We write, for instance,
$[\bsM]$ for the homothety class $[\bsLambda]$ of the lattice $\bsLambda$
determined by $\bsGamma\bsM$ and $\bsM\leq V(\lri)$ if
$\bsLambda \leq V(\lri)$.  The following follows trivially from the block
diagonal structure of~$\bsM$.

\begin{lem}\label{lem:com.matrix}
  For $k\in[d]$ and $t,h\in[a]$ let
  $C_{k}^{(th)}\in\Mat_{n_{t} \times n_{h}}(\Gri)$ be as
  in~\eqref{def:Ckth}. Then
\begin{align*}
 \bsM \leq V(\lri) \Leftrightarrow \forall k\in[d]:\bsM C_{k}\leq \bsM\Leftrightarrow \forall 
 h,t\in[a], k\in[d]:\, M_{t}C_{k}^{(th)}\leq M_{h}.
\end{align*}
\end{lem}

Recall that $c$ is the nilpotency class of $\bmcL$. For $h\in[a]$, define the diagonal matrix
\[\delta_{h}:=\textrm{diag}\left((\pi^{c-1})^{(n_{h,1})},\ldots,(\pi)^{(n_{h,c-1})},1^{(n_{h,c})} \right)\in\Mat_{n_{h}}(\lri)\]
and set
\begin{align}
  \bsdelta:=\textrm{diag}(\delta_{1},\ldots,\delta_{a})\in\Mat_{n}(\lri)\label{eq:delta}
\end{align}	
The following is a trivial consequence of Condition
\ref{cond:quiver.hom.matrix}.
\begin{lem}\label{lem:newdelta}
If $c>1$, then $\forall
h,t\in[a],k\in[d]:\delta_{t}C_{k}^{(th)}\delta_{h}^{-1}=\pi
C_{k}^{(th)}$.
\end{lem}

For $h,t\in[a]$, $r\in[n_{t}]$ and $k\in[d]$, write
$(e_{t,r})c_{k}=\sum_{i=1}^{n_{h}}\lambda_{t,r,k}^{i}e_{h,i}$ for
$\lambda_{t,r,k}^{i}\in\Gri$. Then $C^{(th)}_{k}$ satisfies
$(C_{k}^{(th)})_{r,i}=\lambda_{t,r,k}^{i}$ for $r\in[n_{t}]$ and
$i\in[n_{h}]$. Let $\bfY_{h}=(Y_1,\ldots,Y_{n_{h}})$ be independent variables
and set

\[\mcR^{(th)}(\bfY_{h})=\left(\sum_{i=1}^{n_{h}}\lambda_{t,r,k}^{i}Y_{i}\right)_{r,k}\in\Mat_{n_{t}\times
    d}(\Gri[\bfY_h]).\]

Note that $c=1$ if and only if $(\forall h,t\in[a]:\;
\mcR^{(th)}(\bfY_{h})=0)$. In this case, Theorem~\ref{thm:main.refine}
holds (cf.\ Example~\ref{exa:c=1}), so we may assume $c>1$. For
$i\in[n_{h}]$, we write $\alpha_{h}[i]$ for the $i$-th column of a
matrix $\alpha_{h}\in\Gamma_{h}$, so that
$\mcR^{(th)}(\alpha_{h}[i])\in\Mat_{n_{t}\times d}(\lri)$. The
following lemma is verified by a trivial computation.
\begin{lem}\label{lem:congruence}
  For all $h,t\in[a]$, $\alpha_{h}\in\Gamma_{h}$,
  $\Delta\in\Mat_{n_{t}}(\lri)$, and $D_{h}$ as in \eqref{def:D},
$$  \left(\forall k\in[d]:\Delta C_{k}^{(th)}\alpha_{h}\leq
      D_{h}\right)\iff\left(\forall
      i\in[n_{h}]:\Delta\mcR^{(th)}(\alpha_{h}[i])\equiv0\bmod
      (D_{h})_{ii}\right).$$
\end{lem}
We set, for $h\in[a]$
\begin{equation*}\tau(h):=\sum_{\iota\in I_{h}^*}r_{h,\iota},\quad
  \tau(\bsM) :=\sum_{h\in[a]}\tau(h),\quad \tau'(h):=\tau(\bsM)-\tau(h).
\end{equation*}

\begin{pro}\label{pro:m1}
  Given $\bsM$ as in \eqref{def:M}, there exists a unique
  $\wt{m}_{1}=\wt{m}_{1}(\bsM)\in\mathbb{N}_{0}$ such that, for all
  $m\in\mathbb{N}_{0}$,
	\[\bsM\bsdelta^{m}\leq V(\lri)\;\mathrm{if\;and\;only\;if\;}m\geq\widetilde{m}_{1}.\]
	In particular, $\bsM\leq V(\lri)$ if and only if
        $\wt{m}_{1}=0$. Moreover, $\wt{m}_{1}\leq\tau(\bsM).$
\end{pro}
\begin{proof} 
  For $h\in[a],$ write $M_{h}=D_{h}\alpha_{h}^{-1}$ as above. Using
  Lemmas~\ref{lem:com.matrix}, \ref{lem:newdelta}, and~\ref{lem:congruence} we
  obtain
  \begin{align*}
\lefteqn{    \bsM\bsdelta^{m}\leq V(\lri)}\\
    \Leftrightarrow\;&\forall h,t\in[a],k\in[d]:M_{t}\delta_{t}^{m}C_{k}^{(th)}\leq M_{h}\delta_{h}^{m}\nonumber\\
    \Leftrightarrow\;&\forall h,t\in[a],k\in[d]: \pi^{m}M_{t}C_{k}^{(th)}\leq M_{h}\nonumber\\
    \Leftrightarrow\;&\forall h,t\in[a], k\in[d]: \pi^{m}M_{t}C_{k}^{(th)}\alpha_{h}\leq D_{h}\nonumber\\
    \Leftrightarrow\;&\forall h,t\in[a], i\in[n_{h}]:\pi^{m}D_{t}\alpha_{t}^{-1}\mcR^{(th)}(\alpha_{h}[i])\equiv0\bmod (D_{h})_{ii}\nonumber\\
    \Leftrightarrow\;&\forall h,t\in[a], i\in[n_{h}]:\pi^{m}D_{t}\alpha_{t}^{-1}\mcR^{(th)}(\alpha_{h}[i])\pi^{\sum_{i>\iota\in I_{h}}r_{h,\iota}}\equiv0\bmod \pi^{\tau(h)}\nonumber\\
    \Leftrightarrow\;&\forall h,t\in[a]:\pi^{\tau'(h)+m}D_{t}\alpha_{t}^{-1}(\mcR^{(th)}(\alpha_{h}[1])\,|\cdots|\,\mcR^{(th)}(\alpha_{h}[n_{h}]))\cdot\label{eq:globalcong}\\
       		&\quad\diag\left(1^{(di_{h,1})},(\pi^{r_{h,i_{h,1}}})^{d(i_{h,2}-i_{h,1})},\ldots,(\pi^{\sum_{\iota\in I_{h}}r_{h,\iota}})^{(d(n_{h}-i_{h,l_{h}}))}\right)\equiv0\bmod \pi^{\tau(\bsM)}.\nonumber
	\end{align*}
In the last congruence, we may replace $\alpha_{t}^{-1}$ by the
adjunct matrix $\alpha_{t}^{\textrm{adj}}$. Setting, for $i\in[n_{h}]$
and $r\in[n_{t}],$
 \begin{align*}
   \mcR_{(i)}^{(th)}(\alpha_{t},\alpha_{h})& =\alpha_{t}^{\textrm{adj}}\mcR^{(th)}(\alpha_{h}[i]),\\
   v_{ir}^{(th)}(\alpha_{t},\alpha_{h}) &=\min\left\{v\left(\mcR_{(\iota)}^{(th)}(\alpha_{t},\alpha_{h})_{\rho\sigma}\right)\mid\iota\leq i,\rho\geq r,\sigma\in[d]\right\},
 \end{align*}
 and 
 \begin{align*}
   m^{(th)}(\bsM)&=\min\left\{\tau(h),\sum_{r\leq\rho\in I_{t}^{*}}r_{t,\rho}+\sum_{i>\iota\in I_{h}^{*}}r_{h,\iota}+v_{ir}^{(th)}(\alpha_{t},\alpha_{h})\mid\,i\in[n_{h}], r\in[n_{t}]\right\},\\
   m_{1}(\bsM) &=\min_{h,t\in[a]}\{\tau'(h)+m^{(th)}(\bsM)\},
 \end{align*}
 we may rephrase the above equivalence as follows:
\begin{align}
  \bsM\bsdelta^{m}\leq V(\lri)&\Leftrightarrow\forall h,t\in[a]:\,m\geq\tau(\bsM)-\left(\tau'(h)+m^{(th)}(\bsM)\right)\nonumber\\
                                  &\Leftrightarrow m\geq\tau(\bsM)-m_{1}(\bsM)=:\wt{m}_{1}(\bsM).\qedhere
\end{align}
\end{proof}

\begin{dfn}  For a lattice $\bsLambda$ corresponding to a coset $\bsGamma \bsM$, we set
  $\wt{m}_1([\bsLambda])=\wt{m}_{1}(\bsM).$
\end{dfn}


\subsection{$\delta$-equivalence}
Recall the diagonal matrix $\bsdelta$ defined in (\ref{eq:delta}). 

\begin{dfn}\label{def:delta.equiv}
  Lattice classes $[\bsLambda],\,[\bsLambda']\in\bmcV$ are called
  $\bsdelta$-equivalent, written $[\bsLambda]\sim[\bsLambda'],$ if there
  exists $m\in\Z$ such that $[\bsLambda]=[\bsLambda'\bsdelta^{m}].$
\end{dfn}

Just as in \cite{VollIMRN/19}, we will use the terms lattice class for
a homothety class of lattices and $\bsdelta$-class for a
$\sim$-equivalence class of lattice classes. The proof of
Proposition~\ref{pro:m1} shows that in each $\bsdelta$-class $\mcC$
there is a unique lattice class $[\bsLambda_{0}]$ such that
$[\bsLambda_{0}\bsdelta^{m}]\leq V(\lri)$ if and only if
$m\in\mathbb{N}_{0}$. We shall say that $[\bsLambda_{0}]$ generates
$\mcC_{\geq0}$ and write $\bsLambda_{0,\max}$ for the unique maximal
element of $[\bsLambda_{0}]$. Setting
\begin{align*}
  \mcC_{\geq0}=\{[\bsLambda_{0}\bsdelta^{m}]\mid m\geq0\}&=\mcC\cap\SubRep_{V(\lri)},\\
  \mcC_{<0}=\{[\bsLambda_{0}\bsdelta^{m}]\mid m<0\}&=\mcC\setminus\mcC_{\geq0},
\end{align*}
we obtain a partition $\mcC=\mcC_{\geq0}\cup\mcC_{<0}.$ For $h\in[a]$, let
$M_{h,c}\in \Mat_{n_h\times n_{h,c}}(\lri)$ denote the matrix comprising the
last $n_{h,c}$ columns of $M_{h}$; one may also see this as a matrix
representation of the lattice~$\mcL_{h,c}(\lri)$. 

\begin{lem}\label{lem:Mc}
  For almost all prime ideals $\mfp$, the following holds for all
  $\bsM\in\mathrm{GL}_{n}(K_\mfp)\cap\mathrm{Mat}_{n}(\Gri_\mfp)$: if
  $\bsM\leq V(\Gri_\mfp)$, then $v(\bsM)=v(\bsM\bsdelta).$
\end{lem}
\begin{proof}
Recall that $\bsM=\diag\left( M_1,\dots,M_a\right)$. We proceed by
induction on $c$, including the case $c=1$. Indeed, for this base case
the statement holds trivially (and for all $\mfp$) as
$\bsdelta=\Id_n$. Assume thus that $c\geq2$ and that the induction
hypothesis holds.
        
We claim that for almost all $\mfp$ and all $\bsM$, the minimal
$\mfp$-valuation of the entries of $\bsM$ is equal that of the last
block columns $M_{h,c}$ of $M_{h}$: if $\pi$ divides $M_{h,c}$ for all
$h\in[a]$, then it divides the whole matrix $\bsM$. Given $\mfp$, set
$\lri= \Gri_\mfp$.
        
For $h\in[a]$, let
\[M_{h}':=(M_h^{(i,j)})_{i,j\in[2,c]}\in\Mat_{n_{h}-N_{h,c-1}}(\mfo),\]
defining the lattice $\Lambda_{h}\cap Z_{h,c-1}(\mfo)$, where
$Z_{h,c-1}(\mfo)=Z_{h,c-1}\otimes_{\mcO}\mfo$. Also let
$\bsM':=\diag(M_{1}',\ldots,M_{a}')\in\Mat_{n-N_{c-1}}(\mfo)$, defining the
lattice $\bsLambda\cap\bsZ_{c-1}(\mfo)$.  By induction hypothesis, $\bsM'$ has
the desired property that if $\pi$ divides the last block columns $M'_{h,c}$
of $M'_{h}$ for all $h\in[a]$, then $\pi$ divides the whole matrix
$\bsM'$. The statement now follows as in \cite[Lemma 2.6]{VollIMRN/19}.
\end{proof}

Assume from now that $\mfp$ satisfies the conclusions of Lemma
\ref{lem:Mc}. For $\mcC\in\bmcV/\sim$, define
$$\Xi_{\mcC_{\geq0}}(\bfs)= \sum_{\substack{[\bfLambda]\in\mcC_{\geq0}\\\bsLambda=\bsLambda_{\max}}}\prod_{h\in[a]}|\mcL_{h}(\lri):\Lambda_{h}|^{-s_{h}}.$$
Let $\Lambda_{0,\max,h}=\bsLambda_{0,\max}\cap\mcL_{h}(\lri)$. The following is proven just as its analogue~\cite[Corollary
2.7]{VollIMRN/19}.
\begin{cor}
	For every $\mcC\in\bmcV/\sim$,
	\[\Xi_{\mcC_{\geq0}}(\bfs)=\frac{1}{1-q^{-\sum_{h\in[a]}s_{h}\sum_{i=1}^{c-1}N_{h,i}}}\prod_{h\in[a]}|\mcL_{h}(\lri):\Lambda_{0,\max,h}|^{-s_{h}}.\]
\end{cor}
\begin{proof}
	For all $m\in\N_{0}$ we have $\bsLambda_{0,\max}\bsdelta^{m}=(\bsLambda_{0,\max}\bsdelta^{m})_{\max}$ by Lemma \ref{lem:Mc}. Hence 
	\begin{align*}|\bmcL(\lri):\bsLambda_{0,\max}\bsdelta^{m}|&=\prod_{h\in[a]}|\mcL_{h}(\lri):\Lambda_{0,\max,h}\delta_{h}^{m}|\\
		&=\prod_{h\in[a]}|\mcL_{h}(\lri):\Lambda_{0,\max,h}|q^{m\sum_{i=1}^{c-1}N_{h,i}}
	\end{align*}    
	and therefore		
	\begin{align*}
		\Xi_{\mcC_{\geq0}}(\bfs)&=\sum_{\substack{[\bfLambda]\in\mcC_{\geq0}\\\bsLambda=\bsLambda_{\max}}}\prod_{h\in[a]}|\mcL_{h}(\lri):\Lambda_{h}|^{-s_{h}}\\
		&=\sum_{m=0}^{\infty}\prod_{h\in[a]}|\mcL_{h}(\lri):\Lambda_{0,\max,h}\delta_{h}^{m}|^{-s_{h}}\\
		&=\sum_{m=0}^{\infty}\prod_{h\in[a]}|\mcL_{h}(\lri):\Lambda_{0,\max,h}|^{-s_{h}}q^{-s_{h}m\sum_{i=1}^{c-1}N_{h,i}}\\
		&=\left(\sum_{m=0}^{\infty}q^{-m\sum_{h\in[a]}s_{h}\sum_{i=1}^{c-1}N_{h,i}}\right)\prod_{h\in[a]}|\mcL_{h}(\lri):\Lambda_{0,\max,h}|^{-s_{h}}\\\
		&=\frac{1}{1-q^{-\sum_{h\in[a]}s_{h}\sum_{i=1}^{c-1}N_{h,i}}}\prod_{h\in[a]}|\mcL_{h}(\lri):\Lambda_{0,\max,h}|^{-s_{h}}.\qedhere
	\end{align*}
\end{proof} 		
\begin{dfn}
	Given $\bsM =
	\diag(M_{1},\dots,M_{a})\in\mathrm{GL}_{n}(K_\mfp)\cap\mathrm{Mat}_{n}(\lri)$
	as in \eqref{def:M} corresponding to a maximal lattice $\bsLambda$,
	define
	\[m_{2}([\bsLambda])=\min\{v(M_{h,c}) \mid h\in[a]\}.\]
\end{dfn}

We set $\wt{\bsdelta}=\pi^{c-1}\bsdelta^{-1}=\diag(\wt{\delta_{1}},\ldots,\wt{\delta_{a}})$, which gives $\wt{\delta_{h}}=\pi^{c-1}\delta_{h}^{-1}$ for $h\in[a]$. Note that
$\det\wt{\bsdelta}=\pi^{\sum_{i=1}^{c-1}(n-N_{i})}$, $\det\wt{\delta_{h}}=\pi^{\sum_{i=1}^{c-1}(n_{h}-N_{h,i})}$, and
$\mcC_{<0}=\{[\bsLambda_{0}\wt{\bsdelta}^{m}]\mid m>0\}$.
\begin{lem}\label{lem:w.refine}
	With $w([\bsLambda]) := (c-1)\wt{m}_{1}([\bsLambda])-m_{2}([\bsLambda])$ we
	have
	\begin{align*}\Xi_{\mcC_{<0}}(\bfs)&:=\sum_{\substack{[\bsLambda]\in\mcC_{<0}\\\bsLambda=\bsLambda_{\max}}}\prod_{h\in[a]}|\mcL_{h}(\lri):\Lambda_{h}(\lri)|^{-s_{h}}q^{-s_{h}n_{h}w([\bsLambda])}\\
		&=\sum_{m=1}^{\infty}\prod_{h\in[a]}|\mcL_{h}(\lri):\Lambda_{0,\max,h}(\lri)\wt{\delta_{h}}^{m}|^{-s_{h}}.
	\end{align*} 
\end{lem}
\begin{proof}Analogous to \cite[Lemma 2.10]{VollIMRN/19}, we observe that  $v(\bsM\bsdelta^{\wt{m}_{1}([\bsLambda])})=m_2([\bsLambda])$. Hence the matrix $\bsM\pi^{(c-1)\wt{m}_{1}([\bsLambda])-m_{2}([\bsLambda])}$ corresponds to $\bsLambda_{0,\max}\wt{\bsdelta}^{\wt{m}_{1}([\bsLambda])}$, whence for each $h$, $M_{h}\pi^{(c-1)\wt{m}_{1}([\bsLambda])-m_{2}([\bsLambda])}$ corresponds to $\Lambda_{0,\max,h}\wt{\delta}_{h}^{\wt{m}_{1}([\bsLambda])}$. 
\end{proof}
Thus we have
\begin{align*}\Xi_{\mcC_{<0}}(\bfs)	&=\sum_{m=1}^{\infty}\prod_{h\in[a]}|\mcL_{h}(\lri):\Lambda_{0,\max,h}\wt{\delta_{h}}^{m}|^{-s_{h}}\\
	&=\sum_{m=1}^{\infty}\prod_{h\in[a]}|\mcL_{h}(\lri):\Lambda_{0,\max,h}|^{-s_{h}}q^{-s_{h}m\sum_{i=1}^{c-1}(n_{h}-N_{h,i})}\\
	&=\left(\sum_{m=1}^{\infty}q^{-m\sum_{h\in[a]}s_{h}\sum_{i=1}^{c-1}(n_{h}-N_{h,i})}\right)\prod_{h\in[a]}|\mcL_{h}(\lri):\Lambda_{0,\max,h}|^{-s_{h}}\\\
	&=\frac{q^{-\sum_{h\in[a]}s_{h}\sum_{i=1}^{c-1}(n_{h}-N_{h,i})}}{1-q^{-\sum_{h\in[a]}s_{h}\sum_{i=1}^{c-1}(n_{h}-N_{h,i})}}\prod_{h\in[a]}|\mcL_{h}(\lri):\Lambda_{0,\max,h}|^{-s_{h}}.
\end{align*} 

For later reference we record another formula for the invariant~$m_{2}$.
Setting, for $h\in[a]$ and $i\in[n_{h}]$,
\[v_{h,i}^{(2)}(\alpha_{h}):=\min\left\{v\left((\alpha_{h}^{\textrm{adj}})_{\iota\sigma}\right)\mid\iota\geq
  i,\,\sigma\in]N_{h,1},n_{h,c}]\right\}\] and
\[m_{h,2}([\bsLambda]):=\min\left\{\sum_{\iota\in I_{h}^{*}}r_{h,\iota},\sum_{i\leq\iota\in I_{h}^{*}}r_{h,\iota}+v_{h,i}^{(2)}(\alpha_{h})\mid i\in[n_{h}]\right\}\]
we obtain
\[m_{2}([\bsLambda])=\min\left\{m_{h,2}([\bsLambda]) \mid h\in[a]\right\}.\]

Finally, let 
\begin{align*}
	A^{\SubRep}(\bfs)&:=\sum_{\substack{[\bsLambda]\in\bmcV\\\bsLambda=\bsLambda_{\max}}}\prod_{h\in[a]}|\mcL_{h}(\lri):\Lambda_{h}|^{-s_h}q^{-s_{h}n_{h}w([\bsLambda])}\\
	&=\sum_{\substack{[\bsLambda]\in\bmcV\\\bsLambda=\bsLambda_{\max}}}\prod_{h\in[a]}|\mcL_{h}(\lri):\Lambda_{h}|^{-s_h}q^{-s_{h}n_{h}((c-1)\wt{m}_{1}([\bsLambda])-m_{2}([\bsLambda]))}.
\end{align*}
Then one argues as in \cite[p.~19]{VollIMRN/19}
that
\[\zeta_{V(\lri)}(\bfs)=\frac{1}{1-q^{-\sum_{h\in[a]}n_{h}s_{h}}}\frac{1-q^{-\sum_{h\in[a]}s_{h}\sum_{i=1}^{c-1}(n_{h}-N_{h,i})}}{1-q^{-(c-1)\sum_{h\in[a]}s_{h}n_{h}}}A^{\SubRep}(\bfs).\]
Since
\begin{multline*}
	\left.\frac{1}{1-q^{-\sum_{h\in[a]}n_{h}s_{h}}}\frac{1-q^{-\sum_{h\in[a]}s_{h}\sum_{i=1}^{c-1}(n_{h}-N_{h,i})}}{1-q^{-(c-1)\sum_{h\in[a]}s_{h}n_{h}}}\right|_{q\rightarrow q^{-1}}\\=-q^{-\sum_{h\in[a]}s_h\sum_{i=0}^{c-1}N_{h,i}}\frac{1}{1-q^{-\sum_{h\in[a]}n_{h}s_{h}}}\frac{1-q^{-\sum_{h\in[a]}s_{h}\sum_{i=1}^{c-1}(n_{h}-N_{h,i})}}{1-q^{-(c-1)\sum_{h\in[a]}s_{h}n_{h}}},
\end{multline*}
it suffices to show that $A^{\SubRep}(\bfs)$ satisfies the functional equation
\begin{equation}\label{eq:gr(2.9).refine}
	\left.A^{\SubRep}(\bfs)\right|_{q\rightarrow q^{-1}}=(-1)^{n-1}q^{\sum_{h\in[a]}\binom{n_{h}}{2}}A^{\SubRep}(\bfs).
\end{equation}
To compute $A^{\SubRep}(\bfs)$ we need, given a lattice class
$[\bsLambda]\in\bmcV$ with $\nu([\bsLambda])=(\bfI,\bfr) \in
\prod_{h=1}^a \mcP([n_h-1]) \times \left(\N^{\lcard}\times
(\N_0^a\setminus\N^a)\right)$, to keep track of the quantity
\begin{align*}
 \lefteqn{q^{-\sum_{h\in[a]}s_{h}\left(\left(\sum_{\iota\in
       I_{h}^{*}}\iota
     r_{h,\iota}\right)+n_{h}w([\Lambda])\right)}}\\&=q^{-\sum_{h\in[a]}s_{h}\left(\left(\sum_{\iota\in
     I_{h}^{*}}\iota
   r_{h,\iota}\right)+n_{h}\left((c-1)\wt{m}_{1}([\bsLambda])-m_{2}([\bsLambda])\right)\right)}\\ &=q^{-\sum_{h\in[a]}s_h\left(\left(\sum_{\iota\in
     I_{h}^{*}} r_{h,\iota}(\iota +
   n(c-1))\right)-n_{h}\left((c-1)m_{1}([\bsLambda])+m_{2}([\bsLambda])\right)\right)}.
\end{align*}
Here we used~\eqref{eq:index.refine}, Lemma \ref{lem:w.refine}, and the fact that
\begin{equation*}
	\wt{m}_{1}([\bsLambda])=\tau(\bsM)-m_{1}([\bsLambda]) = \left(\sum_{h\in[a]}\sum_{\iota\in I_{h}^{*}} r_{h,\iota}\right)-m_{1}([\bsLambda]).
\end{equation*}
To this end we define, given $(\bfI, \bfr)$ as above, for
$\boldsymbol{m}=(m_1,m_2)\in\mathbb{N}_{0}^{2}$,
\[\mathcal{N}_{\bfI,\bfr,\boldsymbol{m}}^{\SubRep}=\left|\{[\bsLambda]\in\bmcV \mid \nu([\bsLambda])=(\bfI,\bfr),m_{i}([\bsLambda])=m_{i},i\in\{1,2\}\}\right|\]
and set
\begin{multline}\label{eq:gr(2.11).refine}
	A_{\bfI}^{\SubRep}(\bfs)=\\\sum_{\bfr\in\mathbb{N}^{\lcard}\times\left(\N_0^a\setminus\N^a\right)}q^{-\sum_{h\in[a]}s_{h}\sum_{\iota\in I_{h}^{*}} r_{h,\iota}(\iota+n(c-1))}\sum_{\boldsymbol{m}=(m_{1},m_{2})\in\mathbb{N}_{0}^{2}}\mathcal{N}_{\bfI,\bfr,\boldsymbol{m}}^{\SubRep}q^{\sum_{h\in[a]}s_{h}n_{h}\left((c-1)m_{1}+m_{2}\right)},
\end{multline}
so that
$A^{\SubRep}(\bfs)=\sum_{\bfI\in\prod_{h\in[a]}\mcP([n_{h}-1])}A_{\bfI}^{\SubRep}(\bfs)$.

\subsection{$\mfp$-Adic integration}
To establish the functional equation~\eqref{eq:gr(2.9).refine} (and
thus~\eqref{equ:funeq.refine}) we express the function $A^{\SubRep}(s)$ in terms of
suitable substitutions of multivariate functions of the form
$\widetilde{Z^{\textup{gr}}}(\bfs)$
(see~\eqref{def:Z.tilde}). Theorem~\ref{thm:main.refine} will then follow from
Theorem~\ref{thm:grfuneq}.

To this end, let $\bfx=(\bfx_{1},\ldots,\bfx_{a})\in\mfp^{l}\times W_{a}$ with
$\bfx_{h}=(x_{h,i_{h,1}},\ldots,x_{h,i_{h,l_{h}}},x_{h,n_{h}})$ for
$h\in[a]$. (Recall that $W_a = \lri^a\setminus \mfp^a$.) Let further
$\bfy=(\bfy_1,\ldots,\bfy_a)\in\Gamma_{1}\times\cdots\times\Gamma_{a}$ with
$\bfy_h = \left(y_{h,i,j}\right)_{i,j\in[n_h]}$ for $h\in[a]$.  We define sets
of polynomials, for $h,t\in[a]$, and $i\in[n_{h}]$, $r\in[n_{t}]$,
\begin{align*}
	\bff_{i,r}^{(th)}(\bfy)&=\left\{\left(\mathcal{R}^{(th)}_{(\iota)}(\bfy_{t},\bfy_{h})\right)_{\rho\sigma}\mid\iota\leq
        i,\rho\geq
        r,\sigma\in[d]\right\},\\ \bff_{h,i}^{(2)}(\bfy)&=\left\{(\bfy_{h}^{\textrm{adj}})_{\iota\sigma}\mid\iota\geq
        i,\,\sigma\in]N_{h,1},n_{h,c}]\right\},
\end{align*}
and set, for $\bfI\in\prod_{h\in[a]}\mcP([n_{h}-1])$,
\begin{align*}
	\bfg_{\bfI}^{(th)}(\bfx,\bfy)=&\left\{\prod_{\iota\in I_{h}^{*}}x_{h,\iota}\right\}\cup\bigcup_{i\in[n_{h}],\,r\in[n_{t}]}\left(\prod_{\rho\in I_{t}^{*}}x_{t,\rho}^{\delta_{r\leq\rho}}\prod_{\iota\in I_{h}^{*}}x_{h,\iota}^{\delta_{i>\iota}}\right)\bff_{i,r}^{(th)}(\bfy),\\
	\bfg_{\bfI}^{(1)}(\bfx,\bfy)=&\bigcup_{h,t\in[a]}\left(\prod_{k\in[a]\setminus \{h\}}\prod_{\lambda\in I_{k}^{*}}x_{k,\lambda}\right)\bfg_{\bfI}^{(th)}(\bfx,\bfy),\\
\end{align*}
and
\begin{align*}
	\bfg_{\bfI}^{(h,2)}(\bfx,\bfy)=&\left\{\prod_{\iota\in I_{h}^{*}}x_{h,\iota}\right\}\cup\bigcup_{i\in[n_{h}]}\left(\prod_{\iota\in I_{h}^{*}}x_{h,\iota}^{\delta_{i\leq\iota}}\right)\bff_{h,i}^{(2)}(\bfy),\\
	\bfg_{\bfI}^{(2)}(\bfx,\bfy)=&\bigcup_{h\in[a]}\bfg_{\bfI}^{(h,2)}(\bfx,\bfy),
\end{align*}
and, for $\kappa_h\in[n_h]$, 
\begin{equation*}
\bfg_{\kappa_{h},I_{h}^{*}}(\bfx,\bfy) = \left\{\prod_{\iota\in
  I_{h}^{*}}x_{h,\iota}^{\delta_{\iota\kappa_{h}}}\right\}.
\end{equation*}

The ideals generated by the sets $\bff_{i,r}^{(th)}(\bfy)$ and
$\bff_{h,i}^{(2)}(\bfy)$ are all $\bfB(F)$-invariant. Informally
speaking, to see this one needs to check that the ideals generated by
the relevant matrix entries do not change when $\bfy$ is replaced by
an element in the coset $\bfy \bfB(F)$; the entries themselves,
however, may change, of course.

With these data we define the $\mfp$-adic integral
\begin{multline*}
Z_{\bI}^{\SubRep}(\bfs)=Z_{\bI}^{\SubRep}\left((s_{1,\iota_{1}})_{\iota_{1}\in
  I_{1}^{*}},\ldots,(s_{a,\iota_{a}})_{\iota_{a}\in
  I_{a}^{*}},s_{n}^{(1)},s_{n}^{(2)}\right):=
\\ \int_{\mfp^{\lcard}\times
  W_{a}\times\bsGamma}\left\Vert\bfg_{\bfI}^{(1)}(\bfx,\bfy)\right\Vert^{s_{n}^{(1)}}\left\Vert\bfg_{\bfI}^{(2)}(\bfx,\bfy)\right\Vert^{s_{n}^{(2)}}\prod_{\boldsymbol{\kappa}\in\prod_{h=1}^a[n_{h}]}\prod_{h\in[a]}\left\Vert\bfg_{\kappa_{h},I_{h}^{*}}(\bfx,\bfy)\right\Vert^{s_{h,\kappa_{h}}}\left|\mathrm{d}\bfx_{\bI}\right|\left|\mathrm{d}\bfy\right|.
\end{multline*}

Here, $\bfs$ is a vector of complex variables; note, however, that
$s_{h,\kappa_h}$ occurs on the right-hand side if and only if
$\kappa_{h}\in I_{h}^{*}$.

It now remains to prove that, for each $\bI\in\prod_{h\in[a]}\mcP([n_{h}-1]),$
the generating function $A_{\bfI}^{\SubRep}(\bfs)$ is indeed
obtainable from the $\mfp$-adic integral $Z_{\bI}^{\SubRep}(\bfs)$ by a
suitable specialization of the variables $\bfs$. We start by measuring the
sets on which the integrand of $Z_{\bI}^{\SubRep}(\bfs)$ is constant. More
precisely we set, for $\boldsymbol{m}=(m_{1},m_{2})\in\mathbb{N}_{0}^{2}$ and
$\br\in\mathbb{N}^{\lcard}\times (\mathbb{N}_{0}^{a}\setminus\mathbb{N}^{a})$,
\begin{multline*}
  \mu_{\bI,\br,\boldsymbol{m}}^{\SubRep}:=\\\mu\left\{(\bfx,\bfy)\in\mfp^{\lcard}\times W_{a}\times\bsGamma\,\mid\,\forall h\in[a],\iota\in I_{h}^{*}:v(x_{h,\iota})=r_{h,\iota},\boldsymbol{m}(\bfx,\bfy)=(m_1,m_2)\right\},
  \end{multline*}
  where
  $\boldsymbol{m}(\bfx,
  \bfy)=(\boldsymbol{m}(\bfx,\bfy)_{1},\boldsymbol{m}(\bfx,\bfy)_{2})$ and,
  for $h,t\in[a]$,
\begin{align*}
	\boldsymbol{m}(\bfx,\bfy)^{(h,t)}&=\min\left\{\tau(h),\sum_{r\leq\rho\in I_{t}^{*}}v(x_{t,\rho})+\sum_{i>\iota\in I_{h}^{*}}v(x_{h,\iota})+v_{i,r}^{(th)}(\bfy)\,\mid\,i\in[n_{h}],r\in[n_{t}]\right\},\\
	\boldsymbol{m}(\bfx,\bfy)_{1}&=\min_{h,t\in[a]}\left\{\tau'(h)+\boldsymbol{m}(\bfx,\bfy)^{(h,t)}\right\},\\
	\boldsymbol{m}(\bfx,\bfy)_{h,2}&=\min\left\{\tau(h),\sum_{i\leq\iota\in I_{h}^{*}}v(x_{h,\iota})+v_{h,i}^{(2)}(\bfy)\,\mid\,i\in[n_{h}]\right\}\\
	\boldsymbol{m}(\bfx,\bfy)_{2}&=\min_{h\in[a]}\left\{\boldsymbol{m}(\bfx,\bfy)_{h,2}\right\}.
\end{align*}

Then, by design, (cf.\ \eqref{eq:normalised})
\begin{multline}\label{eq:gr(2.14)}
	\wt{Z_{\bI}^{\SubRep}}(\bfs) = \frac{1}{(1-q^{-1})^{\lcard+a}\mu(\bsGamma)}\cdot\\\sum_{\br\in\mathbb{N}_{0}^{\lcard}\times(\mathbb{N}_{0}^{a}\setminus\mathbb{N}^{a})}q^{-\sum_{h\in[a]}\sum_{\iota\in I_{h}^{*}} s_{h,\iota}r_{h,\iota}}\sum_{\boldsymbol{m}=(m_{1},m_{2})\in\mathbb{N}_{0}^{2}}\mu_{\bI,\br,\boldsymbol{m}}^{\SubRep}q^{-s_{n}^{(1)}m_{1}-s_{n}^{(2)}m_{2}}.
\end{multline}

The numbers $\mu_{\bI,\br,\boldsymbol{m}}^{\SubRep}$ are closely related to
the natural numbers $\mathcal{N}_{\bI,\br,\boldsymbol{m}}^{\SubRep}$ we are
looking to control.

\begin{lem}
  \begin{equation}\label{eq:gr(2.15)}\mathcal{N}_{\bI,\br,\boldsymbol{m}}^{\SubRep}=\frac{\binom{\underline{n}}{\bfI}_{q^{-1}}}{(1-q^{-1})^{\lcard+a}\mu(\bsGamma)}\mu_{\bI,\br,\boldsymbol{m}}^{\SubRep}q^{\left(\sum_{h\in[a]}\sum_{\iota\in I_{h}^{*}} r_{h,\iota}(\iota(n_{h}-\iota)+1)\right)}.
	\end{equation}
\end{lem}
\begin{proof}
  Analogous to \cite[Lemma 3.1]{Voll/10}.
\end{proof}
Thus, combining \eqref{eq:gr(2.11).refine}, \eqref{eq:gr(2.15)}, and
\eqref{eq:gr(2.14)}, we obtain
\begin{align*}
	\lefteqn{A_{\bfI}^{\SubRep}(\bfs)}\\&=\sum_{\br\in\mathbb{N}_{0}^{\lcard}\times(\mathbb{N}_{0}^{a}\setminus\mathbb{N}^{a})}q^{-\sum_{h\in[a]}s_{h}\sum_{\iota\in I_{h}^{*}} r_{h,\iota}(\iota+n(c-1))}\sum_{\boldsymbol{m}=(m_{1},m_{2})\in\mathbb{N}_{0}^{2}}\mathcal{N}_{\bfI,\bfr,\boldsymbol{m}}^{\SubRep}q^{\sum_{h\in[a]}s_{h}n_{h}\left((c-1)m_{1}+m_{2}\right)}\\
	&=\frac{\binom{\underline{n}}{\bfI}_{q^{-1}}}{(1-q^{-1})^{\lcard+a}\mu(\bsGamma)}\sum_{\br\in\mathbb{N}_{0}^{\lcard}\times(\mathbb{N}_{0}^{a}\setminus\mathbb{N}^{a})}q^{-\sum_{h\in[a]}s_{h}\sum_{\iota\in I_{h}^{*}} r_{h,\iota}((\iota+n(c-1))-\iota(n_{h}-\iota)-1)}\\
	&\quad\quad\cdot\sum_{\boldsymbol{m}\in\mathbb{N}_{0}^{2}}\mu_{\bfI,\bfr,\boldsymbol{m}}^{\SubRep}q^{\sum_{h\in[a]}s_{h}n_{h}\left((c-1)m_{1}+m_{2}\right)}\\
	&= \binom{\underline{n}}{\bfI}_{q^{-1}}\wt{Z_{\bI}^{\SubRep}}\left(\left(\left( s_{h}(\iota_{h}+n(c-1))-\iota_{h}(n_{h}-\iota_{h})-1 \right)_{\iota_h\in I^*_h}\right)_{h\in[a]},-\sum_{h\in[a]}s_{h}n_{h}(c-1),-\sum_{h\in[a]}s_{h}n_{h}\right).
\end{align*}

The functional equation \eqref{eq:gr(2.9).refine} now follows from
Theorem \ref{thm:grfuneq}. This completes the proof of Theorem
\ref{thm:main.refine}.

\begin{acknowledgements}
  The first author was partially supported by the National Research
  Foundation of Korea (NRF) grant no.\ 2019R1A6A1A10073437, funded by
  the Korean government (MEST). We gratefully acknowledge inspiring
  mathematical discussions with Josh Mag\-li\-one, Markus Reineke,
  Tobias Rossmann, and Marlies Vantomme about the research presented
  in this paper. We are particularly indebted to Josh Maglione, who
  allowed us to include the observations pertaining to
  $\ps$-partitions in Section~\ref{subsec:P-part}; these were
  developed in joint conversations after the first draft of this paper
  was completed. We are very grateful to two anonymous referees, whose
  comments and queries helped us to improve this paper significantly.
\end{acknowledgements}

\def\cprime{$'$}
\providecommand{\bysame}{\leavevmode\hbox to3em{\hrulefill}\thinspace}
\providecommand{\MR}{\relax\ifhmode\unskip\space\fi MR }
\providecommand{\MRhref}[2]{%
  \href{http://www.ams.org/mathscinet-getitem?mr=#1}{#2}
}
\providecommand{\href}[2]{#2}

\end{document}